\documentclass[10pt,a4paper]{amsart}
\usepackage{amsmath}
\usepackage{amsfonts}
\usepackage{amssymb}
\usepackage{graphicx}
\usepackage{graphicx}
\usepackage{caption}
\usepackage{subcaption}
\usepackage{tikz}
\usepackage{mathdots}

\usetikzlibrary{arrows,decorations.pathmorphing,decorations.pathreplacing}

\tikzset{vertex/.style={circle,fill=black,inner sep=1pt,outer sep=2pt},
         yvertex/.style={font=\small,minimum size=6pt},
         xvertex/.style={font=\small,minimum size=8pt},
         mvertex/.style={rectangle,draw=black,thick,inner sep=2pt,outer sep=2pt},
         tvertex/.style={inner sep=1pt,font=\scriptsize},
         unvertex/.style={circle,fill=white,draw=white,inner sep=1pt},
         fill1/.style={fill=black!20,draw=black!20},
         fill2/.style={fill=black!40,draw=black!40},
         fill12/.style={fill=black!60,draw=black!60},
         >=stealth',
         leadsto/.style={-angle 90,decorate,decoration=snake,very thick},
         cut/.style={decorate,decoration=saw,very thick}}
\newcommand{\replacevertex}[3][fill=white,draw=white]
 {
  \node at #2 [#1,circle,inner sep=1pt] {};
  \node #2 at #2 #3;
 }

\newtheorem*{theorem*}{Theorem}
\newtheorem{theorem}{Theorem}
\theoremstyle{plain}
\newtheorem{corollary}[theorem]{Corollary}

\newtheorem{lemma}[theorem]{Lemma}
\newtheorem{proposition}[theorem]{Proposition}

\DeclareMathOperator{\add}{add}

\newcommand{\End}{\operatorname{End}\nolimits}
\newcommand{\Ext}{\operatorname{Ext}\nolimits}
\newcommand{\Hom}{\operatorname{Hom}\nolimits}

\renewcommand{\mod}{\operatorname{mod}\nolimits}

\newcommand{\op}{{\operatorname{op}\nolimits}}

\begin{document}
\title[Special biserial cluster-tilted algebras]{SPECIAL BISERIAL CLUSTER-TILTED ALGEBRAS}
\author{Fedra Babaei}
\author{Yvonne Grimeland}

\address{Institutt for matematiske fag\\ NTNU\\ 7491 Trondheim\\ Norway}
\email{Fedra.Babaei@math.ntnu.no}
\email{Yvonne.Grimeland@math.ntnu.no}

\begin{abstract}
We give a complete description of all special biserial cluster-tilted algebras over a finite dimensional hereditary algebra $H$ over an algebraically closed field $K$.
\end{abstract}

\maketitle

\section*{Introduction} 
The class of cluster-tilted algebras was introduced by Buan, Marsh and Reiten in ~\cite{BMR}. They showed that for any cluster-tilted algebra $\Gamma$ arising from a cluster category $\mathcal C$, the Auslander-Reiten sequences in $\mod \Gamma$ are inherited from the Auslander-Reiten triangles in $\mathcal C$.(Auslander-Reiten is hereafter abbreviated as AR.)

Special biserial algebras were first defined in 1983 by Skowro{\'n}ski and Waschb{\"u}sch in ~\cite{WaschSkow}, where their main result gives a bound on the number of non-isomorphic non-projective indecomposable summands in the middle term of any AR-sequence in the module category of a representation-finite special biserial algebra. In ~\cite{WaldWasch} this bound was shown to hold for all special biserial algebras. They also showed that special biserial algebras of infinite representation type are tame and they found a combinatorial description of the AR-sequences of special biserial algebras.

The idea in this paper is to compare the class of cluster-tilted algebras and the class of special biserial algebras, using the knowledge referred to in the previous paragraphs. Our main result is the following
\begin{theorem*}
Let $H$ be a finite dimensional hereditary algebra over an algebraically closed field $K$ and $T$ be a basic cluster-tilting object in the cluster category $\mathcal{C}_{H}$. Then the cluster-tilted algebra $\Gamma=\End_{\mathcal{C}_H}(T)^{\op}$ is a special biserial algebra if and only if 
\begin{itemize}
\item[a)] $H$ is of type $A$, or
\item[b)] $H$ is of type $\widetilde{A}$, or
\item[c)] $H$ is of type $D$, and any indecomposable summand in $T$ is either an $\alpha$- or a $\gamma$-object. 
\end{itemize}
\end{theorem*} 
We refer the reader to figure ~\ref{figurD} in section ~\ref{seksjonD} for an explanation of $\alpha$- and $\gamma$-objects. The notion of $\alpha$-objects was introduced in ~\cite{BOW}.
  
Cluster-tilted algebras of type $A$ and $\widetilde A$ are known to be gentle by ~\cite{ABGP}, and are hence special biserial. Furthermore, there is a complete classification, up to derived equivalence, of all quivers in the mutation class of $E_6$, $E_7$ and $E_8$ in ~\cite{bhl}, see also ~\cite{Bordino}. It follows from this classification that any cluster-tilted algebra of cluster-tilted type $E_6$, $E_7$ or $E_8$ is not special biserial, however our results in section ~\ref{seksjonE} are obtained using a different method.

The paper has $7$ sections, including the introduction. Section two contains background and necessary notation. In section three we introduce results about special biserial and cluster-tilted algebras that have been crucial for our work, and lastly, it ends with a very useful lemma. In section four, five and six we consider cluster-tilted algebras of cluster-tilted type $D$, $E$ and Euclidean type, respectively. Lastly, section seven contains a short summary.
\section{Background and Notation}
For a finite dimensional hereditary algebra $H$ over an algebraically closed field $K$, \emph{the cluster category} $\mathcal C$ of $H$ is defined to be the category $\mathcal{C}=D^b(H)/\tau^{-1}\left[1\right]$, where $\tau^{-1}$ is the inverse of the AR-translate of $D^b(H)$. 

An object $T$ in $\mathcal C$ is called a \emph{tilting object} if $\Ext^1_{\mathcal C}(T,T)=0$ and $T$ is maximal with respect to this property. A tilting object is called \emph{basic} if all of its direct summands are non-isomorphic. All tilting objects considered in this paper will be basic tilting objects.

By \cite[theorem 3.3]{BMRRT} any basic tilting object $T$ in $\mathcal C$ is induced by a basic tilting module over a hereditary algebra $H'$, where $H'$ is derived equivalent to $H$, and $T$ has $n$ indecomposable direct summands, where $n$ is the number of non-isomorphic simple $H$-modules. Also any tilting module over $H$ induces a basic tilting object in $\mathcal C$. Any almost complete basic tilting object $\bar{T}$ in $\mathcal{C}$ has exactly two complements, and if $M$ is one complement for $\bar{T}$, then there are exchange triangles in $\mathcal C$ (see ~\cite{BMRRT}):$$M^{\ast}\rightarrow B\rightarrow M \rightarrow M^{\ast}[1] \ \ \ \text{and} \ \ \ M\rightarrow B'\rightarrow M^{\ast} \rightarrow M[1],$$where $B\rightarrow M$ (or $M\rightarrow B'$) is a minimal right (left, respectively) $\add\bar{T}$-approximation of $M$.

For any tilting object $T$ in $\mathcal C$ there is a \emph{cluster-tilted algebra} $\Gamma=\End_{\mathcal C}(T)^{\op}$. If $T$ is a tilting object in the cluster category of a path algebra of type $A, D, E, \widetilde{A}, \widetilde{D}$ or $\widetilde{E}$, then we say that $\Gamma$ is a cluster-tilted algebra \emph{of cluster-tilted type} $A, D, E, \widetilde{A}, \widetilde{D}$ or $\widetilde{E}$, respectively.

By ~\cite{BMR2} it is known that if $\Gamma$ is a cluster-tilted algebra of finite representation type and $\Gamma=KQ/I$, then the ideal $I$ is generated by the following minimal relations; for each arrow $i\stackrel{a}{\rightarrow}j$ which is part of exactly one cycle, the shortest path $j\rightarrow i$ is a generating zero-relation in $I$. Further, if $i\stackrel{a}{\rightarrow}j$ is part of two different cycles in $Q$, then there are two shortest paths, $\rho_1$ and $\rho_2$, from $j\rightarrow i$ and $\rho_1+x\rho_2$, for some $x\neq0$ in $K$, is a generating commutativity-relation of $I$. For an arrow $i \stackrel{a}{\rightarrow}j$ a path from $j\rightarrow i$ is called \emph{shortest} if it contains no proper subpath that is a cycle and if the full subquiver generated by the induced oriented cycle does not contain any more arrows.

The notion of quiver mutation was defined in ~\cite{FZ1}. Let $Q$ be a quiver with no loops and no cycles of length 2, then mutation on a vertex $i$ in $Q$ will produce a new quiver $Q^*$ satisfying the following conditions:
\begin{itemize}
\item The vertex $i$ is replaced by a vertex $i^*$ and all other vertices are kept,
\item any arrow incident with $i$ is reversed and all other arrows are kept,
\item if there are $r>0$ arrows $j\rightarrow i$ and $s>0$ arrows $i\rightarrow k$, add $rs$ arrows $j\rightarrow k$ then remove a maximal number of $2$-cycles.
\end{itemize}
Two quivers $Q$ and $Q'$ are said to be mutation equivalent if $Q'$ can be obtained from $Q$ by a finite number of mutations.

\section{AR-sequences and Special Biserial Algebras}\label{seksjonARogSB}
In this section we first give the definition of a special biserial algebra. We then state a result regarding the AR-sequences of special biserial algebras ~\cite{WaldWasch}. This result together with an equivalence from ~\cite{BMR} gives Lemma ~\ref{lemma1}, which will be used in sections ~\ref{seksjonD}, ~\ref{seksjonE} and ~\ref{seksjonTilde}.

A finite dimensional algebra $A$ is called \emph{special biserial},~\cite{WaldWasch}, if $A$ is isomorphic to $KQ/I$ for some quiver $Q$ and some admissible ideal $I$ such that both a) and b) are satisfied:
\begin{itemize}
\item[a)] for each vertex $x$ of $Q$ there are at most two arrows starting in $x$ and at most two arrows ending in $x$
\item[b)] for any arrow $a$ in $Q$ there is at most one arrow $b$ such that $ab\notin I$ and at most one arrow $c$ such that $bc\notin I$ 
\end{itemize}

Let $A$ be a finite dimensional algebra over $K$. For any non-projective $A$-module $Z$ there is an AR-sequence
$$0\rightarrow X \rightarrow \bigoplus_{i=1}^{n} Y_i\rightarrow  Z\rightarrow 0 .$$ In ~\cite{AR} they define $\alpha(Z):=n$ to be the number of indecomposable modules in the middle term of the AR-sequence ending in $Z$ and $\beta(Z)$ to be the number of non-projective indecomposable middle terms in the same sequence. For the algebra $A$ they define $\alpha(A) $
to be the maximal number of indecomposable summands occurring in any AR-sequence in $\mod A$ and $\beta (A)$ to be the maximal number of indecomposable non-projective middle terms occurring in any AR-sequence in $\mod A$. From ~\cite{WaldWasch} we have the following result:
\begin{theorem}\label{WW} \cite{WaldWasch}
We have $\beta (A)\leq 2$ for any special biserial algebra $A$.
\end{theorem}
\noindent For algebras of finite representation type there is the following stronger result.
\begin{theorem}\label{WSkow}\cite{WaschSkow}
Let $A$ be an algebra of finite representation type. Then $\beta(A)\leq 2$ if and only if $A$ special biserial. 
\end{theorem}

If $C$ is a triangulated category with AR-triangles, $X$ an indecomposable object in $C$ and $\chi$ the AR-triangle where $X$ is the third term, then we define the number $\alpha_{C}(X)$ to be the number of indecomposable objects in the second term of $\chi$. We define $\alpha(C)= \sup\{\alpha(X)\mid X \text{ an indecomposable object in } C \}$.

From ~\cite{BMR} we have the equivalence:
\begin{theorem}\label{equivalence}
Let $\mathcal C$ be the cluster category of a hereditary algebra $H$, let $T$ be a tilting object in $\mathcal C$ and let $\Gamma=\End_{\mathcal C}(T)^{\op}$, then $F=\Hom_{\mathcal C}(T,-)$
induces an equivalence $$\mathcal{C}/\add (\tau T)\rightarrow \mod \Gamma .$$ Furthermore, the almost split sequences in $\mod \Gamma\simeq\mathcal{C}/\add (\tau T)$ are induced by almost split triangles in $\mathcal C$.
\end{theorem}
\noindent Note that $F$ sends summands of $T$ to projectives in $\mod \Gamma$.

By combining Theorem ~\ref{WW} and ~\ref{equivalence} above we have the lemma:
\begin{lemma}\label{lemma1}
Let $T$ be a tilting object in the cluster category $\mathcal C$ of a hereditary algebra $H$. If there is an indecomposable object $X$ in $\mathcal C$ with the properties
\begin{itemize}
\item[a)] $\alpha_{\mathcal C}(X)=\alpha_{\mathcal C}(\tau X)=3$ and 
\item[b)] $\Hom_{\mathcal C}(T,Y)\neq 0$ for $Y$ in $\{ X, \tau X ,\tau^2 X\}\bigcup \{Z\mid Z$ a summand in the AR-triangle ending in $\ \tau X$ or $X\}$,
\end{itemize}
then $\beta(\Gamma)>2$ and $\Gamma=\End_{\mathcal C}(T)^{\op}$ is not special biserial. 
\end{lemma}
\begin{proof}
We claim that in $\mod \Gamma$ there is an AR-sequence
$$
\begin{tikzpicture}
\node (0) at (0,3) [xvertex] {$0$};
\node (1) at (2,3) [xvertex] {$\Hom_{\mathcal C}(T,\tau X)$};
\node (2) at (5,2.4) [xvertex] {$\Hom_{\mathcal C}(T,Y_3)$};
\node (3) at (5,3) [xvertex] {$\Hom_{\mathcal C}(T,Y_2)$};
\node (4) at (5,3.6) [xvertex] {$\Hom_{\mathcal C}(T,Y_1)$};
\node (5) at (8,3) [xvertex] {$\Hom_{\mathcal C}(T,X)$};
\node (6) at (10,3) [xvertex] {$0$};
\draw [->] (0)--(1);
\draw [->] (1)--(2);
\draw [->] (1)--(3);
\draw [->] (1)--(4);
\draw [->] (2)--(5);
\draw [->] (3)--(5);
\draw [->] (4)--(5);
\draw [->] (5)--(6);
\end{tikzpicture} 
$$
where $\Hom_{\mathcal C}(T,Y_i)$ is non-projective for $i=1,2,3$.
By assumption a) we have the following two AR-triangles in $\mathcal C$; 
$$
\begin{tikzpicture}
\node (0) at (0,3) [xvertex] {$\tau^2 X$};
\node (1) at (2,2.5) [xvertex] {$\tau Y_3$};
\node (2) at (2,3) [xvertex] {$\tau Y_2$};
\node (3) at (2,3.5) [xvertex] {$\tau Y_1$};
\node (4) at (4,3) [xvertex] {$\tau X$};
\node (5) at (6,2.5) [xvertex] {$Y_3$};
\node (6) at (6,3) [xvertex] {$Y_2$};
\node (7) at (6,3.5) [xvertex] {$Y_1$};
\node (8) at (8,3) [xvertex] {$X$};
\draw [->] (0)--(1);
\draw [->] (0)--(2);
\draw [->] (0)--(3);
\draw [->] (1)--(4);
\draw [->] (2)--(4);
\draw [->] (3)--(4);
\draw [->] (4)--(5);
\draw [->] (4)--(6);
\draw [->] (4)--(7);
\draw [->] (5)--(8);
\draw [->] (6)--(8);
\draw [->] (7)--(8);
\end{tikzpicture} .$$
Let $\mathcal{S}=\{\tau^2 X, \tau X, X\} \bigcup \{ \tau Y_i,Y_i \}_{i=1}^3$. By assumption b) we have for all $Z$ in $\mathcal S$ that $\Ext^1_{\mathcal C}(\tau^{-1}Z,T)\neq 0$, so $\tau^{-1} Z$ is not a summand in $T$. 

Applying the equivalence $F$ from Theorem ~\ref{equivalence} it follows that $\Hom_{\mathcal C}(T,Y_i)$ is not projective in $\mod\Gamma$ since $Y_i$ is not a summand of $T$. Furthermore, since $\tau^{-1}Z$ is not a summand of $T$ for any $Z$ in $\mathcal S$, it follows that $X,Y_1,Y_2,Y_3$ and $\tau X$ are not summands of $\tau T$. Hence $\beta_{\mod \Gamma}(\Hom_{\mathcal{C}}(T,X))=3$, so $\Gamma$ is not special biserial by Theorem ~\ref{WW}.
\end{proof}
This lemma can be generalized further by considering triangles with more than three middle terms in assumption a) and then by adjusting assumption b) accordingly.

\section{Dynkin type $D$}\label{seksjonD}
We are interested in finding all the cluster-tilted algebras of type $D_n$ that are also special biserial. 

Throughout this section we assume that $\mathcal C$ is the cluster category of Dynkin type $D_n$. Let $H=KQ$, where $Q$ is some orientation of $D_n$. Then the AR-quiver of $\mathcal C$ is shown in figure ~\ref{figurD}, where the dashed lines are identified with each other.
\begin{figure}[h!]
\[ \scalebox{.8}{ \begin{tikzpicture}[scale=.5,yscale=-1]
 \foreach \x in {0,...,8}
  \foreach \y in {0,1,3,5,7}
   \node (\y-\x) at (\x*2,\y) [vertex] {};
 \foreach \x in {0,...,7}
  \foreach \y in {2,4,6}
   \node (\y-\x) at (\x*2+1,\y) [vertex] {};
 \foreach \xa/\xb in {0/1,1/2,2/3,3/4,4/5,5/6,6/7,7/8}
  \foreach \ya/\yb in {0/2,1/2,3/4,5/4,7/6}
   {
    \draw [->] (\ya-\xa) -- (\yb-\xa);
    \draw [->] (\yb-\xa) -- (\ya-\xb);
   }
 \foreach \xa/\xb in {0/1,1/2,2/3,3/4,4/5,5/6,6/7,7/8}
  \foreach \ya/\yb in {3/2,5/6}
   {
    \draw [thick,loosely dotted] (\ya-\xa) -- (\yb-\xa);
    \draw [thick,loosely dotted] (\yb-\xa) -- (\ya-\xb);
   }
 \draw [dashed] (0,-.5) -- (0,7.5); 
 \draw [dashed] (16,-.5) -- (16,7.5);
\end{tikzpicture} } \]
\caption{}
\label{figurD}
\end{figure} 

Indecomposable objects appearing in the top two rows of the AR-quiver of $\mathcal{C}_{D_n}$, see figure ~\ref{figurD}, are called $\alpha$-objects ~\cite{BOW} and we define $\gamma$-objects as all indecomposable objects in the bottom row of the AR-quiver of $\mathcal{C}_{D_n}$.

For $n\geq 5$ we define an equivalence $\phi$ on $\mathcal C$ as follows, if $X$ is an indecomposable $\alpha$-object then $\phi X$ is the unique other $\alpha$-object occurring as a summand in the middle term of an AR-triangle together with $X$. For all other indecomposable objects $\phi X =X$.

There is an explicit description of the mutation class of $D_n$ quivers in ~\cite{Vatne}, where the mutation class is divided into four types of quivers. In ~\cite{BOW} this was combined into three types. We will use the latter description, and we state the details needed to discuss which quivers in the mutation class correspond to special biserial cluster-tilted algebras.

First we give a description of the mutation class $\mathcal{M}^A_k$ of $A_k$, see~\cite{BV}; It is the set of all connected quivers satisfying the following four requirements;
\begin{itemize}
\item there are $k$ vertices and each vertex has valency at most four (i.e. it has at most four neighbors),
\item all non-trivial cycles are oriented and of length 3,(such cycles are called 3-cycles),
\item if a vertex has valency four, then two of the arrows belongs to one 3-cycle and the other two belong to another 3-cycle and
\item if a vertex has valency three, then two of the arrows belongs to a 3-cycle and the third arrow does not belong to any 3-cycle.
\end{itemize}
For a quiver $Q$ in $\mathcal{M}^A_k$, a vertex is called a connecting vertex if it has valency at most 2, and if it has valency 2 then it is part of a 3-cycle. We denote by $\mathcal{M}^A$ the union of all $\mathcal{M}^A_k$ for all $k$.

We now describe the three types of quivers occurring in the mutation class of $D_n$. Any cluster-tilted algebra of type $D_n$ has a quiver that is one of the three types in figure \ref{treTypar},
\begin{figure}[h]
        \begin{subfigure}[b]{0.45\textwidth}
                \centering
              \begin{tikzpicture}[baseline=1cm]
										\node (s1) at (1,1)[tvertex] {$\star_2$};
										\node (2) at (2,.5)[tvertex] {$v_1$};
										\node (s2) at (3,1)[tvertex] {$\star_1$};
										\node (3) at (2,1.5) [tvertex] {$v_2$};
										\draw [->] (s2)--(3);
										\draw [->] (3)--(s1);
										\draw [->] (s2)--(2);
										\draw [->] (2)--(s1);
										\draw [->] (s1)--(s2);
										\draw[rotate=0]  (0.6,1) ellipse (24pt and 15pt); 
										\draw[rotate=0] (3.4,1) ellipse (24pt and 15pt); 
										\node (q1) at (0.3,1) [xvertex] {$Q_2$};
										\node (q2) at (3.7,1) [xvertex] {$Q_1$};
							\end{tikzpicture}  
        \caption{Type 1}
        \label{Type 1}
        \end{subfigure}
        ~ 
        \begin{subfigure}[b]{0.45\textwidth}
                \centering
                \begin{tikzpicture}[baseline=1cm]
									  \node (s1) at (1,1) [tvertex]{$\star_2$};
										\node (2) at (2,.5) [tvertex] {$v_1$};
										\node (s2) at (3,1)[tvertex] {$\star_1$};
										\node (3) at (2,1.5) [tvertex] {$v_2$};
										\draw [->] (s2)--(3);
										\draw [->] (3)--(s1);
										\draw [->] (s1)--(2);
										\draw [->] (2)--(s2);	
										\draw[rotate=0]  (0.6,1) ellipse (24pt and 15pt); 
										\draw[rotate=0] (3.4,1) ellipse (24pt and 15pt); 
										\node (q1) at (0.3,1) [xvertex] {$Q_2$};
										\node (q2) at (3.7,1) [xvertex] {$Q_1$};					
									\end{tikzpicture}   
        \caption{Type 2}
        \label{Type 2}
        \end{subfigure}
        
        \begin{subfigure}[b]{0.80\textwidth}
         \centering
           \begin{tikzpicture}
  											\node (8) at (337.5:1cm) [tvertex] {$v_3$};
  											\node (7) at (292.5:1cm)    {};
											  \node (1) at (22.5:1cm) [tvertex] {$v_2$};
  											\node (2) at (67.5:1cm) [tvertex] {$v_1$};
  											\node (3) at (112.5:1cm) [tvertex] {$v_k$}; 
  											\node (4) at (157.5:1cm) [tvertex] {$v_{k-1}$}; 
  											\node (5) at (202.5:1cm) {};
  											\node (s1) at (45:2cm)[tvertex] {$\star_1$};
  											\node (s2) at (90:2cm) [tvertex] {$\star_{k}$};
  											\node (s3) at (135:2cm) [tvertex] {$\star_{k-1}$};
  											\node (s4) at (180:2cm) [tvertex] {$\star$};
  											\node (s7) at (315:2cm) [tvertex] {$\star$};
  											\node (s8) at (360:2cm) [tvertex] {$\star_2$};
  											\draw [->] (2)--(1);	
  											\draw [->] (3)--(2);	
											  \draw [->] (4)--(3);
  											\draw [->] (5)--(4);
  										  \draw [->] (1)--(8);
  											\draw[->] (8)--(7);
  											\draw [thick, loosely dotted] (5) .. controls (225:1cm) and (265:1cm)  .. (7);
 												\draw [->] (1)--(s1);
 												\draw [->] (s1)--(2);
 												\draw [->] (2)--(s2);
 												\draw [->] (s2)--(3);
 												\draw [->] (3) --(s3);
 												\draw [->] (s3)--(4);
 												\draw [->] (4)--(s4);
 												\draw [->] (s4)--(5);
											 \draw [->] (7)--(s7);
											 \draw [->]  (s7)--(8);
 											\draw [->] (8)--(s8);
 											\draw [->] (s8)--(1);
 											\draw[rotate=0] (2.4,0) ellipse (24pt and 15pt);
 											\draw[rotate=45] (0,0) (2.4,0) ellipse (24pt and 15pt);
 											\draw[rotate=90] (0,0) (2.4,0) ellipse (24pt and 15pt);
 											\draw[rotate=135] (0,0) (2.4,0) ellipse (24pt and 15pt);
 											\draw[rotate=-45] (0,0) (2.4,0) ellipse (24pt and 15pt);
 											\draw[rotate=180] (0,0) (2.4,0) ellipse (24pt and 15pt);
 											\node (S1) at (45:2.6cm)[tvertex] {$Q_1$};
 											\node (S2) at (90:2.6cm) [tvertex] {$Q_{k}$};
  										\node (S3) at (135:2.6cm) [tvertex] {$Q_{k-1}$};
  										\node (S4) at (180:2.6cm) [tvertex] {$Q$};
  										\node (S7) at (315:2.6cm) [tvertex] {$Q$};
  										\node (S8) at (360:2.6cm) [tvertex] {$Q_2$};
	        \end{tikzpicture}    
        \caption{Type 3}
        \label{Type 3}
        \end{subfigure}
       \caption{}\label{fig:}
       \label{treTypar}
\end{figure}
\noindent where $\star_i$ is a connecting vertex of $Q_i$, unless $Q_i$ is the empty quiver. Note that $Q_i$ can possibly consist of only one vertex and it can possibly be empty. The valency of $\star_i$ in the quiver as a whole might be greater or equal to three.

Quivers of type $3$ have a full subquiver which is a directed $k$-cycle, called the central cycle of the quiver. For each $i\in\left\{1,\cdots,k\right\}$, if the quiver $Q_i$ (see figure ~\ref{treTypar}) is not empty then there is an oriented $3$-cycle $v_i\rightarrow v_{i+1}\rightarrow \star_i\rightarrow v_i$ which is a full subquiver and where $v_{i+1}\rightarrow \star_i\rightarrow v_i$ is a non-zero path.

For the case $n=4$ it is clear that quivers of type $3$ are also of type $1$ or $2$, and that there are only two different quivers corresponding to special biserial cluster-tilted algebras. 

\begin{lemma}
If $n\geq 5$, then any cluster-tilted algebra with corresponding quiver of type $1$ or $2$ is not special biserial.
\end{lemma}
\begin{proof}
Since $n\geq 5$ then at least one of $Q_1,Q_2$ has two or more vertices.

For type $1$ assume that $Q_1\neq \emptyset$ and that $Q_2=\emptyset$. Then $\star_1$ has valency three or four. If it has valency three, then either all arrows go out of $\star_1$ or there is one arrow going in to $\star_1$ and this arrow is part of two non-zero paths of length two passing through $\star_1$. If the valency of $\star_1$ is four, then there will be three arrows going out of $\star_1$. For $Q_1=\emptyset$ and $Q_2\neq \emptyset$ the arguments will be symmetric. 

If both $Q_1\neq\emptyset$ and $Q_2\neq\emptyset$ then at least one of $\star_1$ and $\star_2$ will have valency either four or five. If $\star_1$ has valency four, then there will be three arrows going out of $\star_1$ or there will be one arrow going in to $\star_1$ and this arrow is part of two non-zero paths of length two passing through $\star_1$. Further, if $\star_1$ has valency five it is clear that there will be three arrows going out of $\star_1$. The arguments for $\star_2$ will be symmetric.

For type $2$ assume that $Q_1\neq\emptyset$ and $Q_2=\emptyset$. Then $\star_1$ will have valency at least three and there is an arrow which is part of two non-zero paths of length two passing through $\star_1$. Similarly for $Q_2\neq\emptyset$, $Q_1=\emptyset$ and for $Q_1\neq\emptyset\neq Q_2$. 
\end{proof} 

\begin{lemma}\label{type3}
If $n\geq 5$, then a cluster-tilted algebra with quiver of type $3$ is special biserial if and only if for each $v_i$ in the central cycle, $Q_i$ is empty or has only one vertex.
\end{lemma}
\begin{proof}
If $Q_i$ has more than one vertex for some $i$, then the valency of $\star_i$ is at least three and since $ v_{i+1}\rightarrow \star_i\rightarrow v_i$ is a non-zero path, there is an arrow which is part of two non-zero paths of length two passing through $\star_i$. Thus the corresponding cluster-tilted algebra is not special biserial. However for all quivers of type 3 where for each $i$ the quiver $Q_i$ either consist of exactly one vertex or is empty, the corresponding cluster-tilted algebras are special biserial. 
\end{proof}

We want to describe all tilting objects $T$ in $\mathcal C$ such that $\Gamma=\End_{\mathcal C}(T)^{\op}$ is special biserial.

It is known by ~\cite[Theorem 4.1]{BOW} that for any quiver $Q$, where $Q$ is the quiver of a cluster-tilted algebra $\End_{\mathcal C}(T)^{\op}$, of type $3$ each vertex $v_i$  in the central cycle corresponds to an indecomposable $\alpha$-object in $T$. Furthermore, if $\beta:v_i\rightarrow v_{i+1}$ is an arrow in the central cycle such that $v_i$ corresponds to the indecomposable object $X$, then $v_{i+1}$ corresponds to $\tau\phi X$ if $Q_i$ is empty and to $\tau^2 X$ if $Q_i$ is not empty. 

For two tilting objects $T$ and $T'$ in $\mathcal C$ it is known by ~\cite{HermA} that $\End_{\mathcal C}(T)^{\op}$ and $\End_{\mathcal C}(T')^{\op}$ are isomorphic if and only if $T'=\phi^{i}\tau^{j} T$ for $i\in\{0,1\},j \in \mathbb Z$.

For an $\alpha$-object $Y$ we define $\gamma(Y)$ to be the unique object in the bottom row of the AR-quiver of $\mathcal{C}_{D_n}$ such that $\gamma(Y)$ has a non-zero map to both $\tau Y$ and $\tau\phi Y$. Uniqueness of $\gamma (Y)$ follows from the next proposition.

\begin{proposition}\label{prop1}
Let $T$ be a cluster-tilting object of type $D_n$, $n\geq 5$ and $\Gamma=\End_{\mathcal C}(T)^{\op}$ with quiver $Q_T$. Let $j$ be a vertex of $Q_T$ such that there is only one arrow $i\rightarrow j$ in $Q_T$ going in to $j$ and such that the summand in $T$ corresponding to $j$ is an $\alpha$-object $Y$ and the summand corresponding to $i$ is $\tau\phi Y$. Then there is an exchange triangle 
$$ \gamma(Y)\rightarrow \tau\phi Y \rightarrow Y \rightarrow \gamma(Y)\left[1\right]. $$
\end{proposition} 
\begin{proof}
Let $T$ be a cluster-tilting object of type $D_n$ satisfying the assumptions in the proposition.  

Fix the orientation of $D_n$ to be
$$ \begin{tikzpicture}
\node (N) at (0,2) [vertex] {};
\node (1) at (1,1.75) [yvertex]{$1$};
\node (2) at (1,2.25) [yvertex]{$2$};
\node (3) at (2,2) [yvertex]{$3$};
\node (4) at (3,2)[yvertex] {$4$};
\node (dots) at (4.2,2)[yvertex] {$\cdots$};
\node (n-1) at (5.7,2)[yvertex] {$n-1$};
\node (n) at (7,2) [yvertex] {$n$};
\replacevertex{(N)}{[yvertex] {$Q':$}}
\draw [->] (1)--(3);
\draw [->] (2)--(3);
\draw [->] (3)--(4);
\draw [->] (4)--(dots);
\draw [->] (dots)--(n-1);
\draw [->] (n-1)--(n);
\end{tikzpicture} $$
and consider $T$ with the embedding of the module category of $KQ'$ in $\mathcal C$,\\
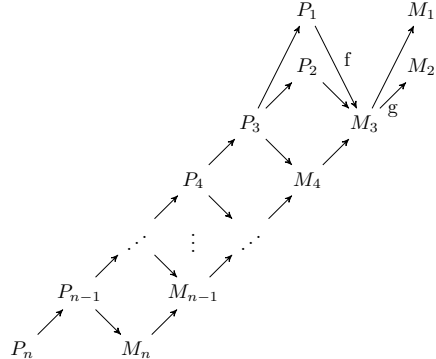
\begin{minipage}{\textwidth}
\[ \scalebox{.75}{\begin{tikzpicture}
 \node (Pn) at (1,1) {$P_n$};
 \node (Pn-1) at (2,2) {$P_{n-1}$};
 \node (dots1) at (3,3) {$\iddots $};
 \node (P4) at (4,4) {$P_4$};
 \node (P3) at (5,5) {$P_3$};
 \node (P1) at (6,7) {$P_1$}; 
 \node (P2) at (6,6) {$P_2$};
 \draw [->] (Pn)--(Pn-1);
 \draw [->] (Pn-1)--(dots1);
 \draw [->] (dots1)--(P4);
 \draw [->] (P4)--(P3);
 \draw [->] (P3)--(P1);
 \draw[->] (P3)--(P2);    
 \node (Mn) at (3,1) {$M_n$};
 \node (Mn-1) at (4,2) {$M_{n-1}$};
 \node (dots2) at (5,3) {$\iddots$};
 \node (M4) at (6,4) {$ M_4$};
 \node (M3) at (7,5) {$M_3 $};
 \node (M2) at (8,6) {$M_2 $};
 \node (M1) at (8,7) {$M_1 $};
 \draw[->] (Mn)--(Mn-1);
 \draw[->] (Mn-1)--(dots2);
 \draw[->] (dots2)--(M4);
 \draw[->] (M4)--(M3);
 \draw[->] (M3)--(M1);
 \draw[->] (M3)--(M2);
 \draw[->] (P1)--(M3);
 \draw[->] (P2)--(M3);
 \draw[->] (P3)--(M4);
 \draw[->] (P4)--(dots2);
 \draw[->] (dots1)--(Mn-1);
 \draw[->] (Pn-1)--(Mn);
 \node[above] at (7.5,5) {g};
 \node[above] at (6.7,5.9) {f};
 \node (dotsM) at (4,3) {$\vdots$};     
\end{tikzpicture} } \]
\end{minipage} 
We wish to exchange the indecomposable summand $Y$ in $T$. By applying the equivalences $\tau$ and $\phi$ to $Y$ an appropriate number of times we get $\tau^{i}\phi^{j}Y=M_2$, and can instead consider how to replace $M_2$ in $T'$, where $T'=\tau^{i}\phi^{j} T$.

By assumption the right $\add T'$-approximation of $M_2$ is the irreducible map $gf: P_1 \rightarrow M_2$ and by calculating the composition series of $P_1$ and $M_2$ it is easy to check that there is a short exact sequence $$0\rightarrow P_n \rightarrow P_1 \stackrel{gf}{\rightarrow} M_2\rightarrow 0 \label{korteksakt},$$
giving the exchange triangle $$P_n\rightarrow P_1 \rightarrow M_2 \rightarrow P_n \left[1\right].$$
\end{proof} 
In the rest of this section we will give an explicit description of the cluster-tilting objects $T$ such that $\End_{\mathcal C}(T)^{\op}$ is special biserial. 
\begin{theorem}
The special-biserial cluster-tilted algebras of type $D_n,\ n \geq 5,$ are those where the corresponding quiver $Q_T$ has the following shape;\newline
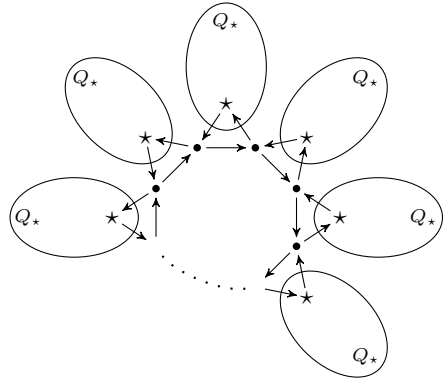
\begin{equation*}
\begin{tikzpicture}
  \node (8) at (337.5:1cm) [vertex] {};
  \node (7) at (292.5:1cm)    {};
  \node (1) at (22.5:1cm) [vertex] {};
  \node (2) at (67.5:1cm) [vertex] {};
  \node (3) at (112.5:1cm) [vertex] {}; 
  \node (4) at (157.5:1cm) [vertex] {}; 
  \node (5) at (202.5:1cm) {};
  \node (s1) at (45:1.5cm)[inner sep=1pt] {$\star$};
  \node (s2) at (90:1.5cm) [inner sep=1pt] {$\star$};
  \node (s3) at (135:1.5cm) [inner sep=1pt] {$\star$};
  \node (s4) at (180:1.5cm) [inner sep=1pt] {$\star$};
  \node (s7) at (315:1.5cm) [inner sep=1pt] {$\star$};
  \node (s8) at (360:1.5cm) [inner sep=1pt] {$\star$};
  \draw [->] (2)--(1);
  \draw [->] (3)--(2);
  \draw [->] (4)--(3);
  \draw [->] (5)--(4);
  \draw [->] (1)--(8);
  \draw[->] (8)--(7);
  \draw [thick, loosely dotted] (5) .. controls (225:1cm) and (265:1cm)  .. (7);
 \draw [->] (1)--(s1);
 \draw [->] (s1)--(2);
 \draw [->] (2)--(s2);
 \draw [->] (s2)--(3);
 \draw [->] (3) --(s3);
 \draw [->] (s3)--(4);
 \draw [->] (4)--(s4);
 \draw [->] (s4)--(5);
 \draw [->] (7)--(s7);
 \draw [->]  (s7)--(8);
 \draw [->] (8)--(s8);
 \draw [->] (s8)--(1);
 \draw[rotate=0] (2,0) ellipse (24pt and 15pt);
 \draw[rotate=45] (0,0) (2,0) ellipse (24pt and 15pt);
 \draw[rotate=90] (0,0) (2,0) ellipse (24pt and 15pt);
 \draw[rotate=135] (0,0) (2,0) ellipse (24pt and 15pt);
 \draw[rotate=-45] (0,0) (2,0) ellipse (24pt and 15pt);
 \draw[rotate=180] (0,0) (2,0) ellipse (24pt and 15pt);
 \node (S1) at (45:2.6cm)[tvertex] {$Q_{\star}$};
 \node (S2) at (90:2.6cm) [tvertex] {$Q_{\star}$};
 \node (S3) at (135:2.6cm) [tvertex] {$Q_{\star}$};
 \node (S4) at (180:2.6cm) [tvertex] {$Q_{\star}$};
 \node (S7) at (315:2.6cm) [tvertex] {$Q_{\star}$};
 \node (S8) at (360:2.6cm) [tvertex] {$Q_{\star}$};
\end{tikzpicture}
\end{equation*}
\newline
where for each $\star$ the quiver $Q_{\star}$ is either empty or one vertex.

The direct summands of $T$ corresponding to the central cycle of $Q_T$ are $\alpha$-objects, and the distribution of all summands of $T$ in the $AR$-quiver is as follows:
if $Y$ is a direct summand of $T$ that is an $\alpha$-object then 
\begin{itemize}
\item if the arrow going out of $Y$ belonging to the central cycle is part of a 3-cycle then it will go to $\tau^{2} Y$ and the last summand of the 3-cycle will be $\gamma(\tau\phi Y)$.
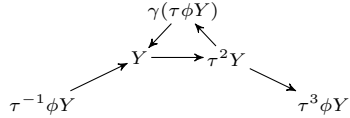
\begin{equation*}
\begin{tikzpicture}
  \node (1) at (22.5:2cm) [tvertex] {$\tau^{3}\phi Y$};
  \node (2) at (67.5:1.5cm) [tvertex] {$ \tau^{2} Y$};
  \node (3) at (112.5:1.5cm) [tvertex] {$Y$};
  \node (4) at (157.5:2cm) [tvertex] {$\tau^{-1}\phi Y$};
  \node (s2) at (90:2cm) [tvertex] {$\gamma(\tau\phi Y)$};
  
  \draw [->] (2)--(1);
  \draw [->] (3)--(2);
  \draw [->] (4)--(3);
  \draw [->] (2)--(s2);
  \draw [->] (s2)--(3);
\end{tikzpicture}
\end{equation*}
\item if the arrow going out of $Y$ belonging to the central cycle is not part of a three cycle then it will go to $\tau\phi Y$
\end{itemize} 
\end{theorem}
\begin{proof}
From Lemma ~\ref{type3} it is clear which quivers correspond to special biserial cluster-tilted algebras. Let $Q_T$ be a cycle, then it is known by~\cite[Theorem 4.1]{BOW} that all the indecomposable summands of $T$ are $\alpha$-objects, and that if $Y$ is one indecomposable summand of $T$ then the other summands of $T$ are given as follows;  
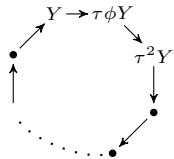
\begin{equation*}
\begin{tikzpicture}
 \node (8) at (337.5:1cm) [vertex] {};
 \node (7) at (292.5:1cm)  [vertex]{};
 \node (1) at (22.5:1cm) [tvertex] {$\tau^2Y$};
  \node (2) at (67.5:1cm) [tvertex] {$\tau \phi Y$};
  \node (3) at (112.5:1cm) [tvertex] {$Y$}; 
  \node (4) at (157.5:1cm) [vertex] {}; 
  \node (5) at (202.5:1cm) {};
  
   \draw [->] (2)--(1);
  \draw [->] (3)--(2);
  \draw [->] (4)--(3);
  \draw [->] (5)--(4);
  \draw [->] (1)--(8);
  \draw[->] (8)--(7);
  \draw [thick, loosely dotted] (5) .. controls (225:1cm) and (265:1cm)  .. (7);
  \end{tikzpicture}
\end{equation*}

If $Q_T$ is not a cycle, then let $k$ be the number of quivers $Q_{\star}$ which are not empty. By doing a sequence of $k$ mutations on a cycle we can get the quiver $Q_T$. The mutations should satisfy proposition ~\ref{prop1}, and in the sequence of mutations a vertex should not be mutated if any of its neighbors in the cycle has been mutated already.
\end{proof}
\begin{theorem}\label{typeDiff}
Let $T$ be a cluster-tilting object in $\mathcal{C}_{D_n}$. Then $\Gamma=\End_{\mathcal C}(T)^{\op}$ is special biserial if and only if all the indecomposable summands of $T$ are $\alpha$- and $\gamma$-objects.
\end{theorem}
\begin{proof}
For cluster-tilted algebras of type $D_4$, it is easy to check that the fact is true. So we assume $n\geq 5$.
It follows from the previous theorem that if $\Gamma$ is special biserial then all the indecomposable summands of $T$ are either $\alpha$-objects or $\gamma$-objects. 

For the opposite direction we consider three cases, if $T$ consists of either $\alpha$-objects or $\gamma$-objects or a combination of the two. Assume that $T$ consists only of $\alpha$-objects. It can be checked that for any $\alpha$-object $Y$, the $\alpha$-objects that are compatible with $Y$ are $\mathcal{S}_Y=\phi Y\bigcup\left\{\tau^{-i}\phi^i Y\right\}_{i=1}^{n-1}$. If both $Y$ and $\phi Y$ are summands in $T$ then there are no more $\alpha$-objects that are compatible with both, and so $T$ can not be completed with only $\alpha$-objects. If $T$ contains $Y$ or $\phi Y$ then in both cases there is exactly one way to complete $T$ by using $\alpha$-objects, if $Y$ is in $T$ then $T=Y\oplus \tau^{-1}\phi Y\oplus \tau^{-2} Y\oplus \tau^{-3}\phi Y\oplus\cdots\oplus\tau^{-(n-1)}\phi^{n-1}Y$ and the completed tilting object $T'$ containing $\phi Y$ is $T'=\phi T$. Both $T$ and $T'$ has one summand from the middle term of each AR-triangle with three middle terms. Applying the equivalence $F$ in Theorem \ref{equivalence} on $\Gamma=\End_{\mathcal C}(T)^{\op}$ and $\Gamma^{'}=\End_{\mathcal C}(T')^{\op}$ will thus give that $\beta(\Gamma)\leq 2$ and $\beta(\Gamma^{'})\leq 2$, so $\Gamma$ and $\Gamma^{'}$ are special biserial by Theorem \ref{WSkow}.

We now look at the case where all the summands in $T$ are $\gamma$-objects. Let $X$ be a $\gamma$-object, then the white part of the following diagram shows the $\Ext$-support of $X$ in $\mathcal C$.
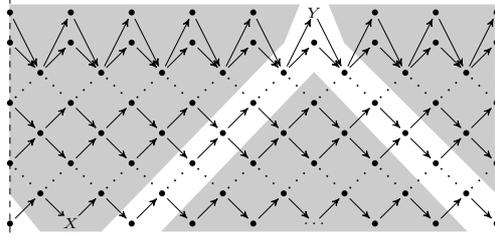
\[ \scalebox{0.8}{ \begin{tikzpicture}[scale=.5,yscale=-1]
 \fill[fill1](0,-0.25)--(9.5,-0.25)--(9,1)--(3,7.25)--(1,7.25)--(0,6)--cycle;
 \fill [fill1] (10.5,-0.25)--(16,-0.25)--(16,6)--(11,1)--cycle;
 \fill [fill1] (5,7.25)--(15,7.25)--(10,2)--cycle;
 \foreach \x in {0,...,8}
  \foreach \y in {0,1,3,5,7}
   \node (\y-\x) at (\x*2,\y) [vertex] {};
 \foreach \x in {0,...,7}
  \foreach \y in {2,4,6}
   \node (\y-\x) at (\x*2+1,\y) [vertex] {};
 \foreach \xa/\xb in {0/1,1/2,2/3,3/4,4/5,5/6,6/7,7/8}
  \foreach \ya/\yb in {0/2,1/2,3/4,5/4,7/6}
   {
    \draw [->] (\ya-\xa) -- (\yb-\xa);
    \draw [->] (\yb-\xa) -- (\ya-\xb);
   }
 \foreach \xa/\xb in {0/1,1/2,2/3,3/4,4/5,5/6,6/7,7/8}
  \foreach \ya/\yb in {3/2,5/6}
   {
    \draw [thick,loosely dotted] (\ya-\xa) -- (\yb-\xa);
    \draw [thick,loosely dotted] (\yb-\xa) -- (\ya-\xb);
   }
 \draw [dashed] (0,-.5) -- (0,7.5); 
 \draw [dashed] (16,-.5) -- (16,7.5);
 
 \replacevertex{(7-1)}{[tvertex][fill1] {$X$}} 
 \replacevertex{(0-5)}{[tvertex] {$Y$}} 
 \replacevertex{(7-5)}{[tvertex][fill1] {$\ldots$}} 
\end{tikzpicture} } \]
From the diagram it is clear that all $\gamma$-objects except $\tau X$ and $\tau^{-1} X$ are  compatible with $X$. Thus the maximal number of $\gamma$-objects in any tilting object in $\mathcal{C}_{D_n}$ is $\left\lfloor \frac{n}{2}\right\rfloor$, and there is no tilting object $T$ in $\mathcal{C}_{D_n}$ consisting only of $\gamma$-objects. 

Finally consider the case when $T$ consists of $\alpha$- and $\gamma$-objects. Fix a $k\in\left\{1,\ldots,\left\lfloor \frac{n}{2}\right\rfloor\right\}$ and let $\bar{T}=\oplus_{i=1}^{k}X_i$ where for $i,j\in\left\{1,\ldots,k\right\}$ the object $X_i$ is a $\gamma$-object such that for all $j\neq i, X_j\notin\left\{ \tau X_i, X_i, \tau^{-1} X_i\right\}$. Choose an $\alpha$-object $Y$ such that $\gamma(Y)\neq X_i$ for $i\in\{1,\ldots,k\}$, then $\bar{T}\oplus Y$ is a rigid object, but not a tilting object. We now try to complete $\bar{T}\oplus Y$ to a tilting object by considering the objects in $\mathcal{S}_Y$. It is clear that also $\bar{T}\oplus Y \oplus \phi Y$ is a rigid object but not a tilting object since $k+2<n$, and there are no more $\alpha$- or $\gamma$-objects that are compatible with $\bar{T}\oplus Y \oplus \phi Y$ so it can not be completed to a tilting object consisting only of $\alpha$- and $\gamma$- objects. For the objects in $\left\{\tau^{-i}\phi^i Y\right\}_{i=1}^{n-1}$ it is clear that for each $X_j$ in $\bar{T}$ there is a unique $i\in \{1,\ldots,n-1\}$ such that $X_j=\gamma(\tau^{-i}\phi^i Y)$. Hence there are $n-1-k$ objects in $\left\{\tau^{-i}\phi^i Y\right\}_{i=1}^{n-1}$ that are compatible with $\bar{T}\oplus Y$, so it can be completed to a tilting object $T$ and any rigid object consisting of only $\gamma$-objects can be completed to a tilting object with $\alpha$-objects in a unique way. Furthermore, it is clear that for $T$ a tilting object consisting of both $\alpha$- and $\gamma$-objects then for each $\alpha$-object $Z$ in $\mathcal{C}$ either $Z$ or $\phi Z$ is in $T$ or one of $\tau^{-1}Z$ and $\tau^{-1}\phi Z$ is in $T$, so by applying the equivalence $F=\Hom_{\mathcal C}(T,-)$ from Theorem ~\ref{equivalence}, $\mod \Gamma$ will not have any AR-sequences with three non-projective middle terms. 
\end{proof}

\section{Dynkin type E}\label{seksjonE}

We show that in the module category of any cluster-tilted algebra $\Gamma$ of type $E_6$, $E_7$ or $E_8$, there is at least one AR-sequence with more than two non-projective middle terms, such that $\Gamma$ is not special biserial by Theorem ~\ref{WSkow}. In \cite{bhl} there is a complete list of all quivers in the respective mutation classes, up to  derived  equivalence. It follow from this that any cluster-tilted algebra of cluster-type $E_6$, $E_7$ or $E_8$ is not special biserial, however we use a different approach in our proofs. 
\subsection{Cluster-tilted algebras of type $E_6$}
The AR-quiver of the cluster category of $E_6$ is a M\"obius band;
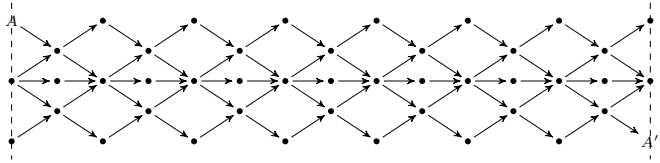
\[ \scalebox{.8}{ \begin{tikzpicture}[scale=.5,yscale=-1]
 \foreach \x in {0,...,7}
  \foreach \y in {1,3,5}
   \node (\y-\x) at (\x*3,\y) [vertex] {};
 \foreach \x in {0,...,6}
  \foreach \y in {2,4}
   \node (\y-\x) at (\x*3+1.5,\y) [vertex] {};
 \foreach \x in {0,...,6}
  \foreach \y in {3}
   \node (a-\x) at (\x*3+1.5,\y)[vertex] {};
  
  \replacevertex{(1-0)}{[tvertex] {$A$}}
  \replacevertex{(5-7)}{[tvertex] {$A'$}} 
 \foreach \xa/\xb in {0/1,1/2,2/3,3/4,4/5,5/6,6/7}
  \foreach \ya/\yb in {1/2,3/2,3/4,5/4}
   {
    \draw [->] (\ya-\xa) -- (\yb-\xa);
    \draw [->] (\yb-\xa) -- (\ya-\xb);
   }
 \foreach \xa/\xb in {0/1,1/2,2/3,3/4,4/5,5/6,6/7}
  \foreach \ya/\yb in {3/3}
  {
   \draw [->] (\ya-\xa) -- (a-\xa);
   \draw [->]  (a-\xa) -- (\yb-\xb);
  }
 \draw [dashed] (0,0.4) -- (0,5.6); 
 \draw [dashed] (21,0.4) -- (21,5.6);
\end{tikzpicture} } \]
where $A$ and $A'$ are identified.
\begin{proposition}\label{E6prop}
Let $T$ be a cluster-tilting object in the cluster category $\mathcal{C}_{E_6}$ and let $\Gamma=\End_{\mathcal{C}_{E_6}}(T)^{\op}$, then $\Gamma$ is not special biserial.
\end{proposition}
\begin{proof}
First we show that $T$ can not consist entirely of indecomposable summands lying in the outermost $\tau$-orbit of the AR-quiver of $\mathcal C_{E_6}$, then there are three $\tau$-orbits/cases left to consider.

Let $M$ be an indecomposable summand of $T$ lying in the outermost $\tau$-orbit of the AR-quiver of $\mathcal C_{E_6}$. The grey areas in figure ~\ref{E6_P1} show which indecomposable objects are in the $\Ext$-support of $M$.
%
%
%
%
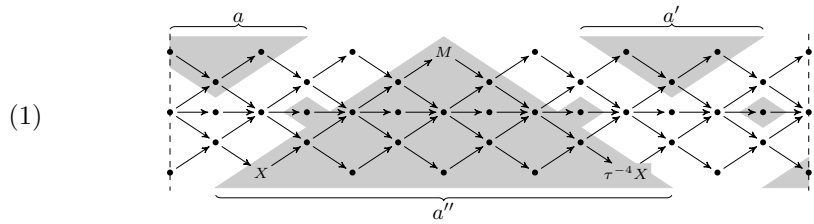
\begin{equation}
\label{E6_P1}
\scalebox{.8}{ \begin{tikzpicture}[baseline=-1.7cm,scale=.5,yscale=-1] 
\fill [fill1] (9,0.5)--(10.5,1.5)--(12,2.5)--(13.5,3.5)--(15,4.5)--(16.5,5.5)--(1.5,5.5)--cycle;
\fill[fill1] (0,0.5)--(4.5,0.5)--(1.5,2.5)--(0,1.5)--cycle;
\fill[fill1] (13.5,0.5)--(19.5,0.5)--(16.5,2.5)--cycle;
\fill[fill1] (18.75,3)--(19.5,2.5)--(20.25,3)--(19.5,3.5)--cycle;
\fill[fill1] (12.75,3)--(13.5,2.5)--(14.25,3)--(13.5,3.5)--cycle;
\fill[fill1] (3.75,3)--(4.5,2.5)--(5.25,3)--(4.5,3.5)--cycle;
\fill[fill1] (19.5,5.5)--(21,4.5)--(21,5.5)--cycle;
 \foreach \x in {0,...,7}
  \foreach \y in {1,3,5}
   \node (\y-\x) at (\x*3,\y) [vertex] {};
 \foreach \x in {0,...,6}
  \foreach \y in {2,4}
   \node (\y-\x) at (\x*3+1.5,\y) [vertex] {};
 \foreach \x in {0,...,6}
  \foreach \y in {3}
   \node (a-\x) at (\x*3+1.5,\y)[vertex] {};
  
   \replacevertex{(1-3)}{[tvertex][fill1] {$M$}}
   \replacevertex{(5-1)}{[tvertex][fill1] {$X$}}  
   \replacevertex{(5-5)}{[tvertex][fill1] {$\tau^{-4}X$}}  
 \foreach \xa/\xb in {0/1,1/2,2/3,3/4,4/5,5/6,6/7}
  \foreach \ya/\yb in {1/2,3/2,3/4,5/4}
   {
    \draw [->] (\ya-\xa) -- (\yb-\xa);
    \draw [->] (\yb-\xa) -- (\ya-\xb);
   }
 \foreach \xa/\xb in {0/1,1/2,2/3,3/4,4/5,5/6,6/7}
  \foreach \ya/\yb in {3/3}
  {
   \draw [->] (\ya-\xa) -- (a-\xa);
   \draw [->]  (a-\xa) -- (\yb-\xb);
  }
 \draw [dashed] (0,0.4) -- (0,5.6); 
 \draw [dashed] (21,0.4) -- (21,5.6);
 \draw [decorate,decoration=brace] (0,.25) -- node [above] {$a$} (4.5,.25);
 \draw [decorate,decoration=brace] (13.5,.25) -- node [above] {$a'$} (19.5,.25);
 \draw [decorate,decoration={brace,mirror}] (1.5,5.75) -- node [below] {$a''$} (16.5,5.75);
\end{tikzpicture} } 
 \end{equation}
%
%
%
%
Assume that all the indecomposable summands of $T$ lie in the outermost $\tau$-orbit together with $M$. It is clear that the two objects in the outermost $\tau$-orbit under the bracket marked $a$ are not compatible with each other, so at most one of these can be in $T$. The exact same argument applies for the bracket marked $a'$. Note that $M$ and $\tau^{-4}M$ are not compatible, so $X$ and $\tau^{-4}X$ are not compatible, see figure ~\ref{E6_P1}. Hence at most two can be chosen from the bracket marked $a''$ such that they are compatible with each other and with $M$. This gives at most five non-isomorphic indecomposable summands in $T$, but this is not sufficiently many for $T$ to be a complete tilting object. 

Now let $M$ be in the $\tau$-orbit indicated in figure \ref{E6_P3}. The grey areas show which indecomposable summands are compatible with $M$. It is clear that all indecomposable summands of a tilting object $T$ containing $M$, must lie in the grey areas of figure \ref{E6_P3}.
%
%
%
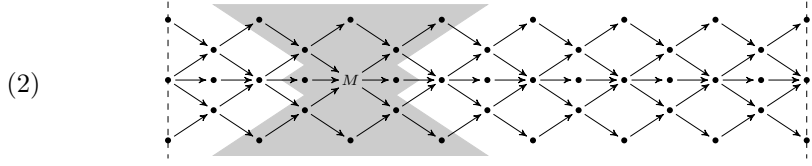
\begin{equation}
\label{E6_P3}
\scalebox{.8}{ \begin{tikzpicture}[baseline=-1.7cm,scale=.5,yscale=-1]
\fill [fill1] (1.5,0.5)--(10.5,0.5)--(6,3.5)--cycle;
\fill[fill1](1.5,5.5)--(6,2.5)--(10.5,5.5)--cycle;
\fill[fill1](6.75,3)--(7.5,2.5)--(8.25,3)--(7.5,3.5)--cycle;
\fill[fill1](3.75,3)--(4.5,2.5)--(5.25,3)--(4.5,3.5)--cycle;

 \foreach \x in {0,...,7}
  \foreach \y in {1,3,5}
   \node (\y-\x) at (\x*3,\y) [vertex] {};
 \foreach \x in {0,...,6}
  \foreach \y in {2,4}
   \node (\y-\x) at (\x*3+1.5,\y) [vertex] {};
 \foreach \x in {0,...,6}
  \foreach \y in {3}
   \node (a-\x) at (\x*3+1.5,\y)[vertex] {};
   
   \replacevertex{(3-2)}{[tvertex][fill1] {$M$}} 
 \foreach \xa/\xb in {0/1,1/2,2/3,3/4,4/5,5/6,6/7}
  \foreach \ya/\yb in {1/2,3/2,3/4,5/4}
   {
    \draw [->] (\ya-\xa) -- (\yb-\xa);
    \draw [->] (\yb-\xa) -- (\ya-\xb);
   }
 \foreach \xa/\xb in {0/1,1/2,2/3,3/4,4/5,5/6,6/7}
  \foreach \ya/\yb in {3/3}
  {
   \draw [->] (\ya-\xa) -- (a-\xa);
   \draw [->]  (a-\xa) -- (\yb-\xb);
  }
 \draw [dashed] (0,0.4) -- (0,5.6); 
 \draw [dashed] (21,0.4) -- (21,5.6);
\end{tikzpicture} }
\end{equation}
%
%
From figure \ref{E6_P3} it is easily seen that there are at least two $AR$-triangles with three middle terms such that all the indecomposable objects in the triangles are not grey and such that the end term of the first triangle is the starting term of the second triangle. By Lemma~\ref{lemma1} we thus have that $\Gamma=\End_{\mathcal C}(T)^{\op}$ is not special biserial if $T$ has a summand from this $\tau$-orbit.

Now assume that $M$ is in the second outermost $\tau$-orbit. As above, the grey areas of figure ~\ref{E6_P2} show which vertices correspond to indecomposable objects in $\mathcal C_{E_6}$ compatible with $M$. 
%
%
%
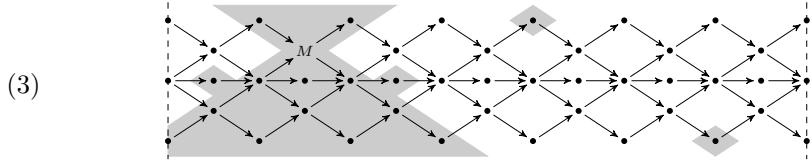
\begin{equation}
\label{E6_P2}
\scalebox{.8}{ \begin{tikzpicture}[baseline=-1.7cm,scale=.5,yscale=-1]
\fill [fill1] (4.5,1.5)--(10.5,5.5)--(0,5.5)--(0,4.5)--cycle;
\fill[fill1](1.5,0.5)--(7.5,0.5)--(4.5,2.5)--cycle;
\fill[fill1](0.75,3)--(1.5,2.5)--(2.25,3)--(1.5,3.5)--cycle;
\fill[fill1](6.75,3)--(7.5,2.5)--(8.25,3)--(7.5,3.5)--cycle;
\fill[fill1](11.25,1)--(12,0.5)--(12.75,1)--(12,1.5)--cycle;
\fill[fill1](17.25,5)--(18,4.5)--(18.75,5)--(18,5.5)--cycle;
 \foreach \x in {0,...,7}
  \foreach \y in {1,3,5}
   \node (\y-\x) at (\x*3,\y) [vertex] {};
 \foreach \x in {0,...,6}
  \foreach \y in {2,4}
   \node (\y-\x) at (\x*3+1.5,\y) [vertex] {};
 \foreach \x in {0,...,6}
  \foreach \y in {3}
   \node (a-\x) at (\x*3+1.5,\y)[vertex] {};
   
   \replacevertex{(2-1)}{[tvertex][fill1] {$M$}} 
 \foreach \xa/\xb in {0/1,1/2,2/3,3/4,4/5,5/6,6/7}
  \foreach \ya/\yb in {1/2,3/2,3/4,5/4}
   {
    \draw [->] (\ya-\xa) -- (\yb-\xa);
    \draw [->] (\yb-\xa) -- (\ya-\xb);
   }
 \foreach \xa/\xb in {0/1,1/2,2/3,3/4,4/5,5/6,6/7}
  \foreach \ya/\yb in {3/3}
  {
   \draw [->] (\ya-\xa) -- (a-\xa);
   \draw [->]  (a-\xa) -- (\yb-\xb);
  }
 \draw [dashed] (0,0.4) -- (0,5.6); 
 \draw [dashed] (21,0.4) -- (21,5.6);
\end{tikzpicture} } 
\end{equation}
%
%
From figure \ref{E6_P2} it is clear that also for this $\tau$-orbit there are at least two AR-triangles with three middle terms such that all the indecomposables are not in the grey areas, and such that the end term of the first triangle is the starting term of the second triangle. By Lemma~\ref{lemma1} we thus have that $\Gamma=\End_{\mathcal{C}_{E_6}}(T)^{\op}$ is not special biserial if $T$ has a summand from this $\tau$-orbit.

Finally assume that $M$ is in the $\tau$-orbit shown in figure ~\ref{E6_P7}.
%
%
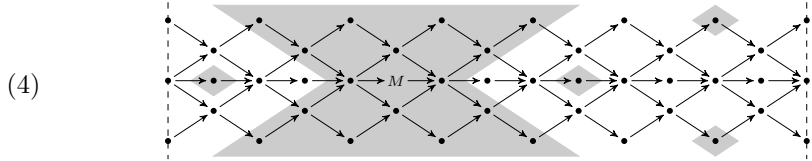
\begin{equation}
\label{E6_P7}
 \scalebox{.8}{ \begin{tikzpicture}[baseline=-1.7cm,scale=.5,yscale=-1]
\fill [fill1] (1.5,0.5)--(13.5,0.5)--(9.75,3)--(13.5,5.5)--(1.5,5.5)--(5.25,3)--cycle;
\fill[fill1](0.75,3)--(1.5,2.5)--(2.25,3)--(1.5,3.5)--cycle;
\fill[fill1](12.75,3)--(13.5,2.5)--(14.25,3)--(13.5,3.5)--cycle;
\fill[fill1](17.25,1)--(18,0.5)--(18.75,1)--(18,1.5)--cycle;
\fill[fill1](17.25,5)--(18,4.5)--(18.75,5)--(18,5.5)--cycle;
 \foreach \x in {0,...,7}
  \foreach \y in {1,3,5}
   \node (\y-\x) at (\x*3,\y) [vertex] {};
 \foreach \x in {0,...,6}
  \foreach \y in {2,4}
   \node (\y-\x) at (\x*3+1.5,\y) [vertex] {};
 \foreach \x in {0,...,6}
  \foreach \y in {3}
   \node (a-\x) at (\x*3+1.5,\y)[vertex] {};
   
   \replacevertex{(a-2)}{[tvertex][fill1] {$M$}} 
 \foreach \xa/\xb in {0/1,1/2,2/3,3/4,4/5,5/6,6/7}
  \foreach \ya/\yb in {1/2,3/2,3/4,5/4}
   {
    \draw [->] (\ya-\xa) -- (\yb-\xa);
    \draw [->] (\yb-\xa) -- (\ya-\xb);
   }
 \foreach \xa/\xb in {0/1,1/2,2/3,3/4,4/5,5/6,6/7}
  \foreach \ya/\yb in {3/3}
  {
   \draw [->] (\ya-\xa) -- (a-\xa);
   \draw [->]  (a-\xa) -- (\yb-\xb);
  }
 \draw [dashed] (0,0.4) -- (0,5.6); 
 \draw [dashed] (21,0.4) -- (21,5.6);
\end{tikzpicture} } 
\end{equation}
%
%
For this $\tau$-orbit there are exactly two AR-triangles where all the indecomposable do not lie in the grey areas, and such that the ending term of the first triangle is the starting term of the second triangle. Thus $\Gamma=\End_{\mathcal C}(T)^{\op}$ is not special biserial if $T$ has a summand from this $\tau$-orbit.
\end{proof}
%
%
%
%
%
\subsection{Cluster-tilted algebras of type $E_7$ and $E_8$}
We begin this subsection by stating a proposition from ~\cite{BOW}. The next proposition follows from applying Theorem 4.9 ~\cite{IY} and is stated explicitly in ~\cite{BOW}.

We get the category $M^{\perp}_{\mathcal{C}}=\{ X\in\mathcal C \mid \Hom_{\mathcal{C}}(X,M\left[1\right])=0\}/(M)$ by taking the subcategory of $\mathcal{C}$ consisting of all objects without extensions with $M$ and then taking the quotient by the ideal consisting of all maps factoring through $M$.

\begin{proposition}\label{BOWprop}
Let $\mathcal{C}$ be the cluster category of a hereditary algebra $H$, and let $M\in\mathcal C$ be indecomposable and rigid. Then $M^{\perp}_{\mathcal C}$ is equivalent to the cluster category $\mathcal{C'}$ of a hereditary algebra $H'$ with exactly one less isomorphism class of simple objects than $H$. If $M$ is a shift of an indecomposable projective module $He$, then $M^{\perp}_{\mathcal C}$ is the cluster category of $H/HeH$.
\end{proposition}

\begin{corollary}\label{BOWcor}~\cite{IY}(see also~\cite{BOW}).
Let $M$ be a rigid indecomposable object in the cluster category $\mathcal C$ of $H$, then there is a bijection between cluster-tilting objects in $\mathcal C$ containing $M$ and cluster-tilting objects in $\mathcal{C}'=\mathcal{C}_{H'}$, where $H'$ is a hereditary algebra with exactly one less isomorphism class of simple objects than $H$ and $\mathcal{C}'$ is equivalent to $M^{\perp}_{\mathcal C}$.
\end{corollary}
This means that for any cluster-tilting object $T$ in $\mathcal C$ having an indecomposable summand $M$, where $M$ is the shift of an indecomposable projective module $He$, then removing $M$ from $T$ gives an object which is a cluster-tilting object in a cluster category equivalent to the cluster category $\mathcal{C}'=M_{\mathcal C}^{\perp}=\mathcal{C}_{H/HeH}$.

The AR-quiver of the cluster category of $E_7$ is a cylinder:
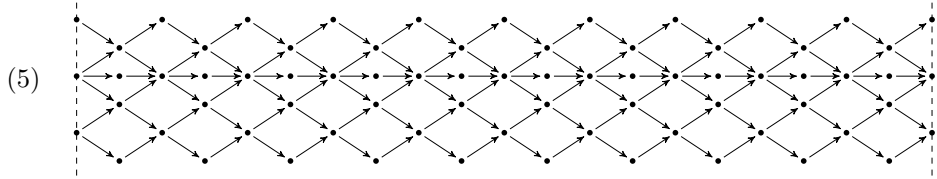
\begin{equation}
 \scalebox{.75}{ \begin{tikzpicture}[baseline=-1.7cm,scale=.5,yscale=-1]
 \foreach \x in {0,...,10}
  \foreach \y in {1,3,5}
   \node (\y-\x) at (\x*3,\y) [vertex] {};
 \foreach \x in {0,...,9}
  \foreach \y in {2,4,6}
   \node (\y-\x) at (\x*3+1.5,\y) [vertex] {};
 \foreach \x in {0,...,9}
  \foreach \y in {3}
   \node (a-\x) at (\x*3+1.5,\y) [vertex] {};
 \foreach \xa/\xb in {0/1,1/2,2/3,3/4,4/5,5/6,6/7,7/8,8/9,9/10}
  \foreach \ya/\yb in {1/2,3/2,3/4,5/4,5/6}
   {
    \draw [->] (\ya-\xa) -- (\yb-\xa);
    \draw [->] (\yb-\xa) -- (\ya-\xb);
   }
   \foreach \xa/\xb in {0/1,1/2,2/3,3/4,4/5,5/6,6/7,7/8,8/9,9/10}
    \foreach \ya/\yb in {3/3}
    {
     \draw [->] (\ya-\xa)--(a-\xa);
     \draw [->] (a-\xa)--(\yb-\xb);   
    }
 \draw [dashed] (0,0.4) -- (0,6.6); 
 \draw [dashed] (30,0.4) -- (30,6.6);
\end{tikzpicture} } 
\end{equation}
\begin{proposition}\label{E7prop}
Let $T$ be a tilting object in the cluster category $\mathcal C$ of $E_7$ and let $\Gamma=\End_{\mathcal C}(T)^{\op}$, then $\Gamma$ is not special biserial.
\end{proposition}
\begin{proof}
We show for every $\tau$-orbit that if $T$ has a summand from this orbit, then $\Gamma$ is not special biserial. First we enumerate the vertices of $E_7$;
$$ 
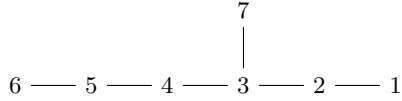
\begin{tikzpicture}
\node (6) at (0,0) [yvertex]{$6$};
\node (5) at (1,0) [yvertex]{$5$};
\node (4) at (2,0) [yvertex]{$4$};
\node (3) at (3,0)[yvertex] {$3$};
\node (2) at (4,0)[yvertex] {$2$};
\node (1) at (5,0)[yvertex] {$1$};
\node (7) at (3,1) [yvertex] {$7$};
\draw [-] (6)--(5);
\draw [-] (5)--(4);
\draw [-] (4)--(3);
\draw [-] (3)--(2);
\draw [-] (2)--(1);
\draw [-] (3)--(7);
\end{tikzpicture} $$
Let $e_6$ be the idempotent corresponding to vertex $6$ and let $M=P_6 \left[i\right]$, $i\in\left\{0,\ldots,10\right\}$. By proposition \ref{BOWprop} we have $M^{\perp}_{\mathcal{C}_{E_7}}=\mathcal{C}_{E_6}$, and by Corollary \ref{BOWcor} it is clear that for any cluster-tilting object $T$ in $\mathcal{C}_{E_7}$ where $P_6 \left[i\right]$ is an indecomposable summand then $\Gamma=\End_{{\mathcal C}_{E_7}}(T)^{\op}$ has a factor algebra that is not special biserial. Thus for any tilting object $T$ with $M$ an indecomposable summand, the corresponding cluster-tilted algebra $\Gamma$ is not special biserial.

Assume that $T$ has a summand $M$ in the $\tau$-orbit of $P_5$. In figure \ref{E7_P5} the grey areas indicate which objects in $\mathcal{C}_{E_7}$ are compatible with $M$.
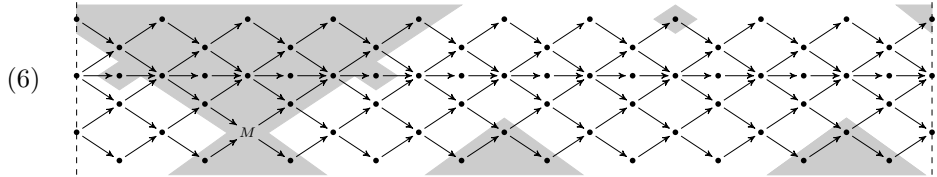
\begin{equation}\label{E7_P5}
\scalebox{.75}{ \begin{tikzpicture}[baseline=-1.7cm,scale=.5,yscale=-1]
 \fill [fill1] (21,0.5) -- (21.75,1) -- (21,1.5) -- (20.25,1) -- cycle;
 \fill [fill1] (10.5,2.5) -- (11.25,3) -- (10.5,3.5) -- (9.75,3) -- cycle;
 \fill [fill1] (1.5,2.5) -- (2.25,3) -- (1.5,3.5) -- (0.75,3) -- cycle;
 \fill [fill1] (29.75,6.5) -- (24.25,6.5) -- (27,4.5) -- cycle;
 \fill [fill1] (17.75,6.5) -- (12.25,6.5) -- (15,4.5) -- cycle;
 \fill [fill1] (29.75,6.5) -- (24.25,6.5) -- (27,4.5) -- cycle;
 \fill [fill1] (3.25,6.5) -- (8.75,6.5)-- (6,4.5)  -- cycle;
 \fill [fill1] (6,5.5) -- (13.5,0.5) -- (0,0.5) -- (0,1.5) --  cycle;
 \fill [fill1] (30,0.5) -- (30,1.5) -- (28.75,0.5) --  cycle;
 \foreach \x in {0,...,10}
  \foreach \y in {1,3,5}
   \node (\y-\x) at (\x*3,\y) [vertex] {};
 \foreach \x in {0,...,9}
  \foreach \y in {2,4,6}
   \node (\y-\x) at (\x*3+1.5,\y) [vertex] {};
 \foreach \x in {0,...,9}
  \foreach \y in {3}
   \node (a-\x) at (\x*3+1.5,\y) [vertex] {};
   
   \replacevertex{(5-2)}{[tvertex][fill1] {$M$}}
 \foreach \xa/\xb in {0/1,1/2,2/3,3/4,4/5,5/6,6/7,7/8,8/9,9/10}
  \foreach \ya/\yb in {1/2,3/2,3/4,5/4,5/6}
   {
    \draw [->] (\ya-\xa) -- (\yb-\xa);
    \draw [->] (\yb-\xa) -- (\ya-\xb);
   }
   \foreach \xa/\xb in {0/1,1/2,2/3,3/4,4/5,5/6,6/7,7/8,8/9,9/10}
    \foreach \ya/\yb in {3/3}
    {
     \draw [->] (\ya-\xa)--(a-\xa);
     \draw [->] (a-\xa)--(\yb-\xb);   
    }
 \draw [dashed] (0,0.4) -- (0,6.5); 
 \draw [dashed] (30,0.4) -- (30,6.5);
\end{tikzpicture} } 
\end{equation}
From figure \ref{E7_P5} it is clear that there are at least two AR-triangles with three middle terms such that all the indecomposable objects in the triangles are not grey and such that the end term of the first triangle is the starting term of the second triangle. Thus it follows by Lemma~\ref{lemma1} that $\Gamma=\End(T)^{\op}$ is not special biserial if $T$ has a summand from this $\tau$-orbit.

Furthermore, assume that $T$ has an indecomposable summand $M$ in the $\tau$-orbit of either $P_4, (P_3, P_2, P_1$ or $P_7$), then the grey areas in the figure ~\ref{E7_P4}, (~\ref{E7_P3},~\ref{E7_P2}, ~\ref{E7_P1} and ~\ref{E7_P7}, respectively) shows which objects in $\mathcal C$ are compatible with the respective summand $M$.
 By the same argument as above we see that in all the cases there are at least two AR-triangles with three middle terms such that all the indecomposable objects in the triangles are not grey and such that the end term of the first triangle is the starting term of the second triangle. It follows from Lemma~\ref{lemma1} that $\Gamma$ is not special biserial if $T$ has a summand from any of these $\tau$-orbits.
%
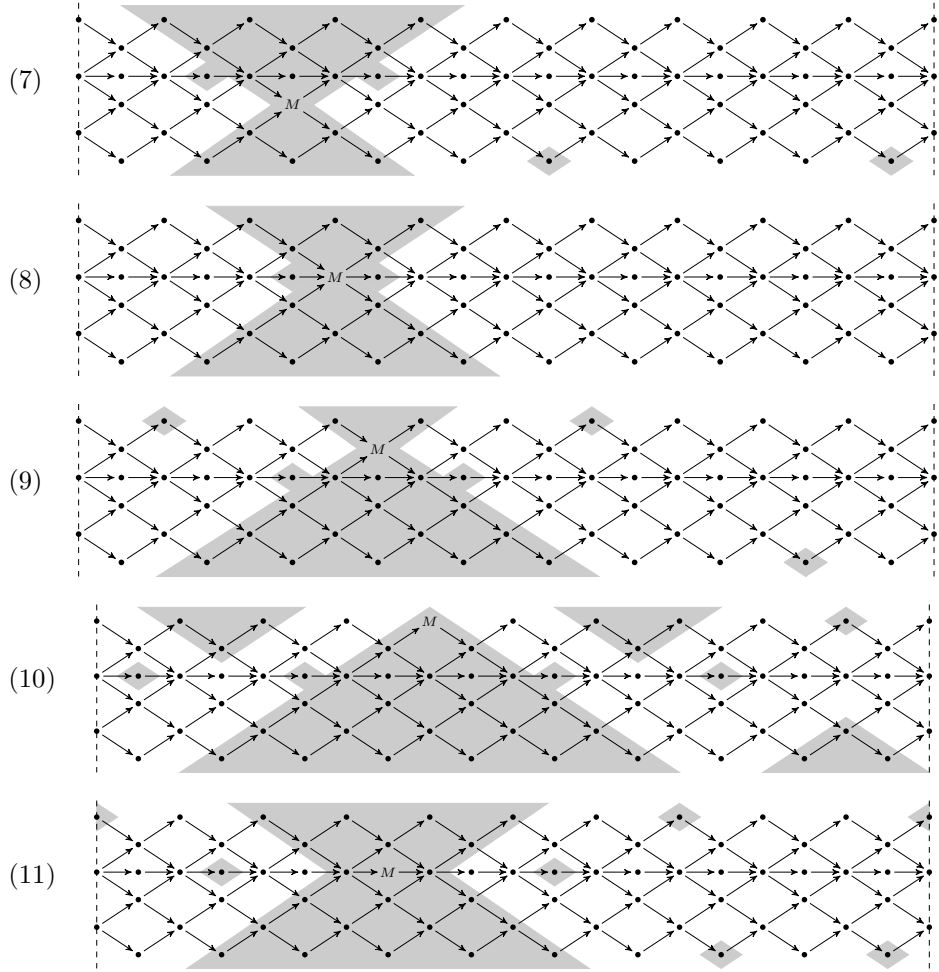
\begin{equation}\label{E7_P4}
 \scalebox{.75}{ \begin{tikzpicture}[baseline=-1.7cm,scale=.5,yscale=-1]
 \fill [fill1] (28.5,6.5) -- (29.25,6) -- (28.5,5.5) -- (27.75,6) -- cycle;
 \fill [fill1] (16.5,6.5) -- (17.25,6) -- (16.5,5.5) -- (15.75,6) -- cycle; 
 \fill [fill1] (4.5,2.5) -- (5.25,3) -- (4.5,3.5) -- (3.75,3) -- cycle;
 \fill [fill1] (10.5,2.5) -- (11.25,3) -- (10.5,3.5) -- (9.75,3) -- cycle;
 \fill [fill1] (11.75,6.5) -- (3.25,6.5) -- (7.5,3.5) -- cycle;
 \fill [fill1] (1.5,0.5) -- (13.5,0.5) -- (7.5,4.5) -- cycle;
 \foreach \x in {0,...,10}
  \foreach \y in {1,3,5}
   \node (\y-\x) at (\x*3,\y) [vertex] {};
 \foreach \x in {0,...,9}
  \foreach \y in {2,4,6}
   \node (\y-\x) at (\x*3+1.5,\y) [vertex] {};
 \foreach \x in {0,...,9}
  \foreach \y in {3}
   \node (a-\x) at (\x*3+1.5,\y) [vertex] {};
   
   \replacevertex{(4-2)}{[tvertex][fill1] {$M$}}
 \foreach \xa/\xb in {0/1,1/2,2/3,3/4,4/5,5/6,6/7,7/8,8/9,9/10}
  \foreach \ya/\yb in {1/2,3/2,3/4,5/4,5/6}
   {
    \draw [->] (\ya-\xa) -- (\yb-\xa);
    \draw [->] (\yb-\xa) -- (\ya-\xb);
   }
   \foreach \xa/\xb in {0/1,1/2,2/3,3/4,4/5,5/6,6/7,7/8,8/9,9/10}
    \foreach \ya/\yb in {3/3}
    {
     \draw [->] (\ya-\xa)--(a-\xa);
     \draw [->] (a-\xa)--(\yb-\xb);   
    }
 \draw [dashed] (0,0.4) -- (0,6.5); 
 \draw [dashed] (30,0.4) -- (30,6.5);
\end{tikzpicture} } 
\end{equation}
%
\begin{equation}\label{E7_P3}
\scalebox{.75}{ \begin{tikzpicture}[baseline=-1.7cm,scale=.5,yscale=-1]
 \fill [fill1] (7.5,2.5) -- (8.25,3) -- (7.5,3.5) -- (6.75,3) -- cycle;
 \fill [fill1] (10.5,2.5) -- (11.25,3) -- (10.5,3.5) -- (9.75,3) -- cycle;
 \fill [fill1] (3.25,6.5) -- (14.75,6.5) -- (9,2.5) -- cycle;
 \fill [fill1] (4.5,0.5) -- (13.5,0.5) -- (9,3.5) -- cycle;
 \foreach \x in {0,...,10}
  \foreach \y in {1,3,5}
   \node (\y-\x) at (\x*3,\y) [vertex] {};
 \foreach \x in {0,...,9}
  \foreach \y in {2,4,6}
   \node (\y-\x) at (\x*3+1.5,\y) [vertex] {};
 \foreach \x in {0,...,9}
  \foreach \y in {3}
   \node (a-\x) at (\x*3+1.5,\y) [vertex] {};
   
   \replacevertex{(3-3)}{[tvertex][fill1] {$M$}}
 \foreach \xa/\xb in {0/1,1/2,2/3,3/4,4/5,5/6,6/7,7/8,8/9,9/10}
  \foreach \ya/\yb in {1/2,3/2,3/4,5/4,5/6}
   {
    \draw [->] (\ya-\xa) -- (\yb-\xa);
    \draw [->] (\yb-\xa) -- (\ya-\xb);
   }
   \foreach \xa/\xb in {0/1,1/2,2/3,3/4,4/5,5/6,6/7,7/8,8/9,9/10}
    \foreach \ya/\yb in {3/3}
    {
     \draw [->] (\ya-\xa)--(a-\xa);
     \draw [->] (a-\xa)--(\yb-\xb);   
    }
 \draw [dashed] (0,0.4) -- (0,6.5); 
 \draw [dashed] (30,0.4) -- (30,6.5);
\end{tikzpicture} } 
\end{equation}
%
\begin{equation}\label{E7_P2}
 \scalebox{.75}{ \begin{tikzpicture}[baseline=-1.7cm,scale=.5,yscale=-1]
 \fill [fill1] (2.25,1) -- (3,0.5) -- (3.75,1) -- (3,1.5) -- cycle;
 \fill [fill1] (17.25,1) -- (18,0.5) -- (18.75,1) -- (18,1.5) -- cycle;
 \fill [fill1] (7.5,2.5) -- (8.25,3) -- (7.5,3.5) -- (6.75,3) -- cycle;
 \fill [fill1] (13.5,2.5) -- (14.25,3) -- (13.5,3.5) -- (12.75,3) -- cycle;
 \fill [fill1] (2.75,6.5) -- (18.25,6.5) -- (10.5,1.5) -- cycle;
 \fill [fill1] (7.75,0.5) -- (13.25,0.5) -- (10.5,2.5) -- cycle;
 \fill [fill1] (24.75,6) -- (25.5,5.5) -- (26.25,6) -- (25.5,6.5) -- cycle;
 \foreach \x in {0,...,10}
  \foreach \y in {1,3,5}
   \node (\y-\x) at (\x*3,\y) [vertex] {};
 \foreach \x in {0,...,9}
  \foreach \y in {2,4,6}
   \node (\y-\x) at (\x*3+1.5,\y) [vertex] {};
 \foreach \x in {0,...,9}
  \foreach \y in {3}
   \node (a-\x) at (\x*3+1.5,\y) [vertex] {};
   
   \replacevertex{(2-3)}{[tvertex][fill1] {$M$}}
 \foreach \xa/\xb in {0/1,1/2,2/3,3/4,4/5,5/6,6/7,7/8,8/9,9/10}
  \foreach \ya/\yb in {1/2,3/2,3/4,5/4,5/6}
   {
    \draw [->] (\ya-\xa) -- (\yb-\xa);
    \draw [->] (\yb-\xa) -- (\ya-\xb);
   }
   \foreach \xa/\xb in {0/1,1/2,2/3,3/4,4/5,5/6,6/7,7/8,8/9,9/10}
    \foreach \ya/\yb in {3/3}
    {
     \draw [->] (\ya-\xa)--(a-\xa);
     \draw [->] (a-\xa)--(\yb-\xb);   
    }
 \draw [dashed] (0,0.4) -- (0,6.5); 
 \draw [dashed] (30,0.4) -- (30,6.5);
\end{tikzpicture} } 
\end{equation}
%
\begin{equation}\label{E7_P1}
 \scalebox{.73}{ \begin{tikzpicture}[baseline=-1.7cm,scale=.5,yscale=-1]
 \fill[fill1] (26.25,1) -- (27,0.5) -- (27.75,1) -- (27,1.5) -- cycle;
 \fill [fill1] (1.5,2.5) -- (2.25,3) -- (1.5,3.5) -- (0.75,3) -- cycle;
 \fill [fill1] (7.5,2.5) -- (8.25,3) -- (7.5,3.5) -- (6.75,3) -- cycle;
 \fill [fill1] (16.5,2.5) -- (17.25,3) -- (16.5,3.5) -- (15.75,3) -- cycle;
 \fill [fill1] (22.5,2.5) -- (23.25,3) -- (22.5,3.5) -- (21.75,3) -- cycle;
 \fill [fill1] (3,6.5) -- (21,6.5) -- (12,0.5) -- cycle;
  \fill[fill1] (1.5,0.5) -- (7.5,0.5) -- (4.5,2.5) -- cycle;
  \fill[fill1] (16.5,0.5) -- (22.5,0.5) -- (19.5,2.5) -- cycle;
  \fill[fill1] (24,6.5) -- (27,4.5) -- (30,6.5) -- cycle;
 \foreach \x in {0,...,10}
  \foreach \y in {1,3,5}
   \node (\y-\x) at (\x*3,\y) [vertex] {};
 \foreach \x in {0,...,9}
  \foreach \y in {2,4,6}
   \node (\y-\x) at (\x*3+1.5,\y) [vertex] {};
 \foreach \x in {0,...,9}
  \foreach \y in {3}
   \node (a-\x) at (\x*3+1.5,\y) [vertex] {};
   
   \replacevertex{(1-4)}{[tvertex][fill1] {$M$}}
 \foreach \xa/\xb in {0/1,1/2,2/3,3/4,4/5,5/6,6/7,7/8,8/9,9/10}
  \foreach \ya/\yb in {1/2,3/2,3/4,5/4,5/6}
   {
    \draw [->] (\ya-\xa) -- (\yb-\xa);
    \draw [->] (\yb-\xa) -- (\ya-\xb);
   }
   \foreach \xa/\xb in {0/1,1/2,2/3,3/4,4/5,5/6,6/7,7/8,8/9,9/10}
    \foreach \ya/\yb in {3/3}
    {
     \draw [->] (\ya-\xa)--(a-\xa);
     \draw [->] (a-\xa)--(\yb-\xb);   
    }
 \draw [dashed] (0,0.4) -- (0,6.5); 
 \draw [dashed] (30,0.4) -- (30,6.5);
\end{tikzpicture} }
\end{equation}
%
\begin{equation}\label{E7_P7}
 \scalebox{.73}{ \begin{tikzpicture}[baseline=-1.7cm,scale=.5,yscale=-1]
 \fill [fill1] (0,0.5) -- (0,1.5) -- (0.75,1) -- cycle;
 \fill [fill1] (30,0.5) -- (30,1.5) -- (29.25,1) -- cycle;
 \fill [fill1] (4.5,2.5) -- (5.25,3) -- (4.5,3.5) -- (3.75,3) -- cycle;
 \fill [fill1] (16.5,2.5) -- (17.25,3) -- (16.5,3.5) -- (15.75,3) -- cycle;
 \fill [fill1] (21,0.5) -- (21.75,1) -- (21,1.5) -- (20.25,1) -- cycle;
 \fill [fill1] (3.25,6.5) -- (17.75,6.5) -- (10.5,1.5) -- cycle;
 \fill [fill1] (4.75,0.5) -- (16.25,0.5) -- (10.5,4.5) -- cycle;
 \fill[fill1] (21.75,6) -- (22.5,5.5) -- (23.25,6) -- (22.5,6.5) -- cycle;
 \fill[fill1] (27.75,6) -- (28.5,5.5) -- (29.25,6) -- (28.5,6.5) -- cycle;
 \foreach \x in {0,...,10}
  \foreach \y in {1,3,5}
   \node (\y-\x) at (\x*3,\y) [vertex] {};
 \foreach \x in {0,...,9}
  \foreach \y in {2,4,6}
   \node (\y-\x) at (\x*3+1.5,\y) [vertex] {};
 \foreach \x in {0,...,9}
  \foreach \y in {3}
   \node (a-\x) at (\x*3+1.5,\y) [vertex] {};
   
   \replacevertex{(a-3)}{[tvertex][fill1] {$M$}}
 \foreach \xa/\xb in {0/1,1/2,2/3,3/4,4/5,5/6,6/7,7/8,8/9,9/10}
  \foreach \ya/\yb in {1/2,3/2,3/4,5/4,5/6}
   {
    \draw [->] (\ya-\xa) -- (\yb-\xa);
    \draw [->] (\yb-\xa) -- (\ya-\xb);
   }
   \foreach \xa/\xb in {0/1,1/2,2/3,3/4,4/5,5/6,6/7,7/8,8/9,9/10}
    \foreach \ya/\yb in {3/3}
    {
     \draw [->] (\ya-\xa)--(a-\xa);
     \draw [->] (a-\xa)--(\yb-\xb);   
    }
 \draw [dashed] (0,0.4) -- (0,6.5); 
 \draw [dashed] (30,0.4) -- (30,6.5);
\end{tikzpicture} } 
\end{equation}
\end{proof}
%
The AR-quiver of the cluster category of $E_8$ is a cylinder and has the following shape:
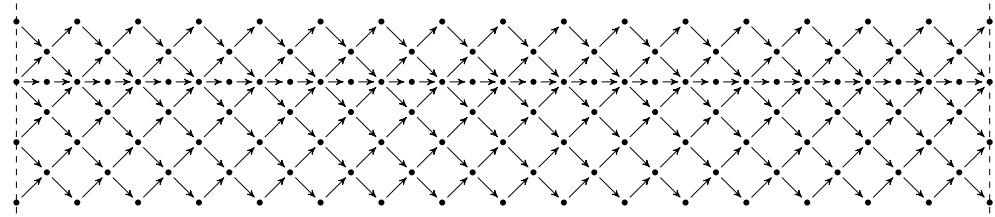
\[ \scalebox{.8}{ \begin{tikzpicture}[scale=.5,yscale=-1]
 \foreach \x in {0,...,16}
  \foreach \y in {1,3,5,7}
   \node (\y-\x) at (\x*2,\y) [vertex] {};
 \foreach \x in {0,...,15}
  \foreach \y in {2,4,6}
   \node (\y-\x) at (\x*2+1,\y) [vertex] {};
 \foreach \x in {0,...,15}
  \foreach \y in {3}
   \node (a-\x) at (\x*2+1,\y) [vertex] {};  
 \foreach \xa/\xb in {0/1,1/2,2/3,3/4,4/5,5/6,6/7,7/8,8/9,9/10,10/11,11/12,12/13,13/14,14/15,15/16}
  \foreach \ya/\yb in {1/2,3/2,3/4,5/4,5/6,7/6}
   {
    \draw [->] (\ya-\xa) -- (\yb-\xa);
    \draw [->] (\yb-\xa) -- (\ya-\xb);
   }
 \foreach \xa/\xb in {0/1,1/2,2/3,3/4,4/5,5/6,6/7,7/8,8/9,9/10,10/11,11/12,12/13,13/14,14/15,15/16}
    \foreach \ya/\yb in {3/3}
   {
     \draw [->] (\ya-\xa)--(a-\xa);
     \draw [->] (a-\xa) --(\ya-\xb);
   }     
 \draw [dashed] (0,0.4) -- (0,7.6); 
 \draw [dashed] (32,0.4) -- (32,7.6);
\end{tikzpicture} } \]
\begin{proposition}\label{E8prop}
Let $T$ be a tilting object in the cluster category of $E_8$ and let $\Gamma=\End_{\mathcal C}(T)^{\op}$, then $\Gamma$ is not special biserial.
\end{proposition}
\begin{proof}
We show for every $\tau$-orbit in the AR-quiver of $\mathcal C$ that if $T$ has a summand from this orbit, then $\Gamma$ is not special biserial. First we enumerate the vertices of $E_8$;
$$ 
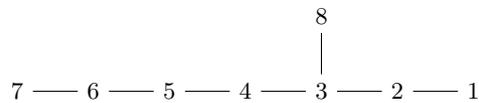
\begin{tikzpicture}
\node (7) at (-1,0)[yvertex]{$7$};
\node (6) at (0,0) [yvertex]{$6$};
\node (5) at (1,0) [yvertex]{$5$};
\node (4) at (2,0) [yvertex]{$4$};
\node (3) at (3,0)[yvertex] {$3$};
\node (2) at (4,0)[yvertex] {$2$};
\node (1) at (5,0)[yvertex] {$1$};
\node (8) at (3,1) [yvertex] {$8$};
\draw [-] (7)--(6);
\draw [-] (6)--(5);
\draw [-] (5)--(4);
\draw [-] (4)--(3);
\draw [-] (3)--(2);
\draw [-] (2)--(1);
\draw [-] (3)--(8);
\end{tikzpicture} $$
Let $e_7$ be the idempotent corresponding to vertex $7$ and $M=P_7 \left[i\right]$, $i\in\left\{0,\ldots,16\right\}$. By proposition~\ref{BOWprop} it is clear that $M^{\perp}_{\mathcal{C}_{E_8}}=\mathcal{C}_{E_7}$, and thus if $M$ is a summand of $T$ then $\Gamma$ has a factor algebra which is not special biserial by Proposition ~\ref{E7prop}. So $\Gamma$ is not special algebra. Further let $e_6$ be the idempotent corresponding to vertex $6$ and 
$M=P_6 \left[i\right]$, $i\in\left\{0,\ldots,16\right\}$. Then by proposition ~\ref{BOWprop}, $M^{\perp}_{\mathcal{C}_{E_8}}=\mathcal{C}_{E_6}\times \mathcal{C}_{A_1}$ and hence if $T$ has a summand from this $\tau$-orbit then by ~\ref{E6prop} and the same argument as above $\Gamma$ is not special biserial.

Furthermore, assume that $T$ has an indecomposable summand $M$ in the $\tau$-orbit of either $P_5,( P_4, P_3, P_2, P_1$ or $P_8$), then the grey areas in the figure ~\ref{E8_P5}, (~\ref{E8_P4}, ~\ref{E8_P3},~\ref{E8_P2}, ~\ref{E8_P1} and ~\ref{E8_P8}, respectively) shows which objects in $\mathcal C$ are compatible with the respective summand $M$. From these figures we see that in all the cases there are at least two AR-triangles with three middle terms such that all the indecomposable objects in the triangles are not grey and such that the end term of the first triangle is the starting term of the second triangle. It follows from Lemma~\ref{lemma1} that $\Gamma$ is not special biserial if $T$ has a summand from any of these $\tau$-orbits.
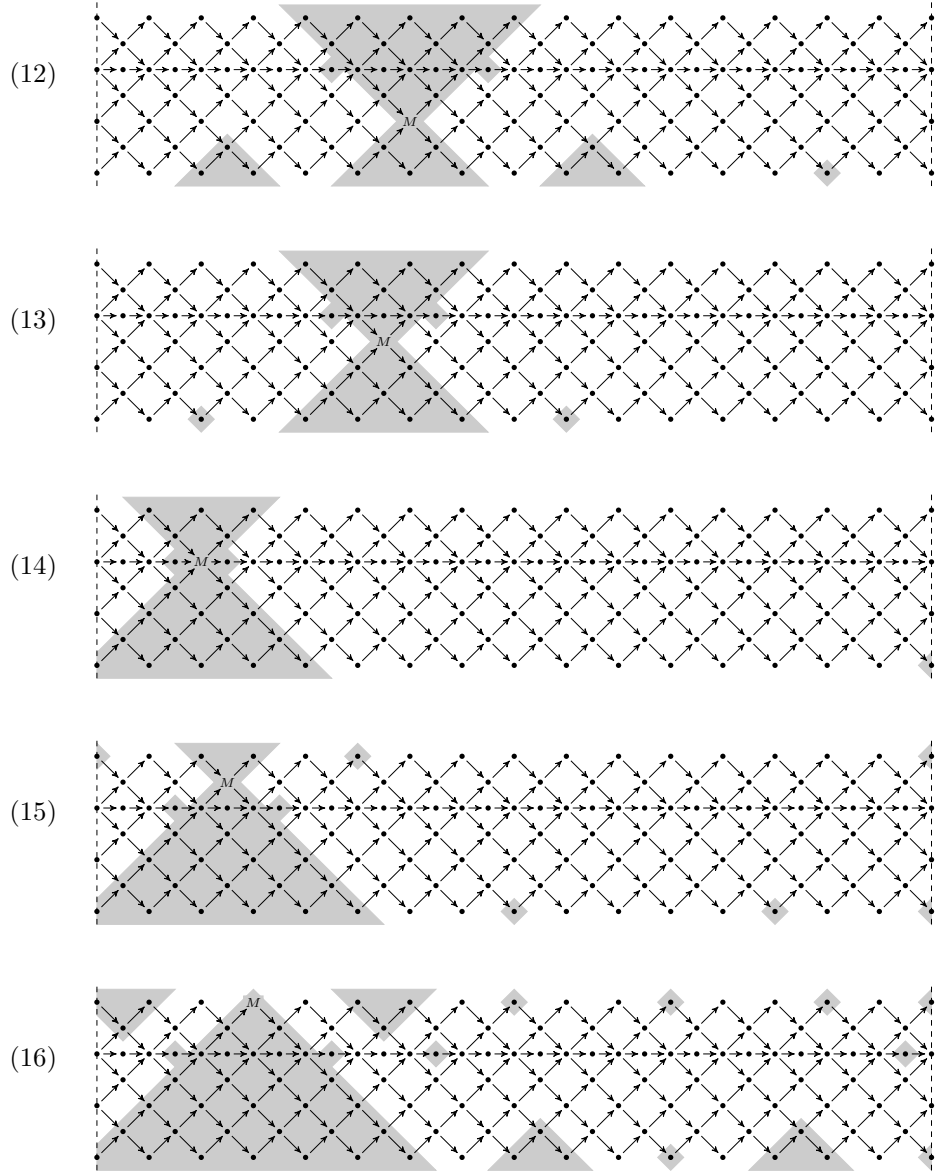
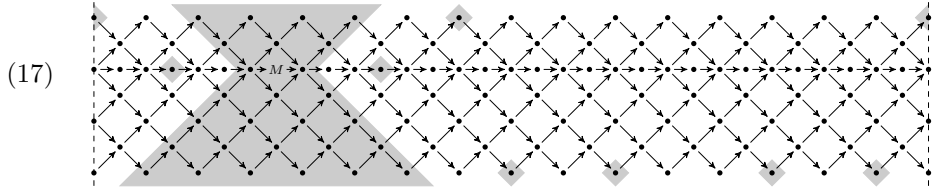
\begin{equation}
 \scalebox{.7}{ \begin{tikzpicture}[baseline=-1.7cm,scale=.49,yscale=-1]\label{E8_P5}
 \fill[fill1] (3,7.5)--(5,5.5)--(7,7.5)--cycle;
 \fill[fill1] (9,7.5)--(12,4.5)--(15,7.5)--cycle;
 \fill[fill1] (7,0.5)--(12,5.5)--(17,0.5)--cycle;
 \fill[fill1] (17,7.5)--(19,5.5)--(21,7.5)--cycle;
  \fill[fill1] (27.5,7)--(28,6.5)--(28.5,7)--(28,7.5)--cycle;
  \fill[fill1] (8.5,3)--(9,2.5)--(9.5,3)--(9,3.5)--cycle;
  \fill[fill1] (14.5,3)--(15,2.5)--(15.5,3)--(15,3.5)--cycle;
 \foreach \x in {0,...,16}
  \foreach \y in {1,3,5,7}
   \node (\y-\x) at (\x*2,\y) [vertex] {};
 \foreach \x in {0,...,15}
  \foreach \y in {2,4,6}
   \node (\y-\x) at (\x*2+1,\y) [vertex] {};
 \foreach \x in {0,...,15}
  \foreach \y in {3}
   \node (a-\x) at (\x*2+1,\y) [vertex] {};
   
   \replacevertex{(5-6)}{[tvertex][fill1] {$M$}}
 \foreach \xa/\xb in {0/1,1/2,2/3,3/4,4/5,5/6,6/7,7/8,8/9,9/10,10/11,11/12,12/13,13/14,14/15,15/16}
  \foreach \ya/\yb in {1/2,3/2,3/4,5/4,5/6,7/6}
   {
    \draw [->] (\ya-\xa) -- (\yb-\xa);
    \draw [->] (\yb-\xa) -- (\ya-\xb);
   }
 \foreach \xa/\xb in {0/1,1/2,2/3,3/4,4/5,5/6,6/7,7/8,8/9,9/10,10/11,11/12,12/13,13/14,14/15,15/16}
    \foreach \ya/\yb in {3/3}
   {
     \draw [->] (\ya-\xa)--(a-\xa);
     \draw [->] (a-\xa) --(\ya-\xb);
   }     
 \draw [dashed] (0,0.4) -- (0,7.6); 
 \draw [dashed] (32,0.4) -- (32,7.6);
\end{tikzpicture} } 
\end{equation}
%
\begin{equation}
 \scalebox{.7}{ \begin{tikzpicture}[baseline=-1.7cm,scale=.49,yscale=-1]\label{E8_P4}
 \fill[fill1] (7,0.5)--(11,4.5)--(15,0.5)--cycle;
 \fill[fill1] (7,7.5)--(11,3.5)--(15,7.5)--cycle;
 \fill[fill1] (3.5,7)--(4,6.5)--(4.5,7)--(4,7.5)--cycle;
 \fill[fill1] (8.5,3)--(9,2.5)--(9.5,3)--(9,3.5)--cycle;
 \fill[fill1] (12.5,3)--(13,2.5)--(13.5,3)--(13,3.5)--cycle;
 \fill[fill1] (17.5,7)--(18,6.5)--(18.5,7)--(18,7.5)--cycle;
 \foreach \x in {0,...,16}
  \foreach \y in {1,3,5,7}
   \node (\y-\x) at (\x*2,\y) [vertex] {};
 \foreach \x in {0,...,15}
  \foreach \y in {2,4,6}
   \node (\y-\x) at (\x*2+1,\y) [vertex] {};
 \foreach \x in {0,...,15}
  \foreach \y in {3}
   \node (a-\x) at (\x*2+1,\y) [vertex] {}; 
   
   \replacevertex{(4-5)}{[tvertex][fill1] {$M$}}
 \foreach \xa/\xb in {0/1,1/2,2/3,3/4,4/5,5/6,6/7,7/8,8/9,9/10,10/11,11/12,12/13,13/14,14/15,15/16}
  \foreach \ya/\yb in {1/2,3/2,3/4,5/4,5/6,7/6}
   {
    \draw [->] (\ya-\xa) -- (\yb-\xa);
    \draw [->] (\yb-\xa) -- (\ya-\xb);
   }
 \foreach \xa/\xb in {0/1,1/2,2/3,3/4,4/5,5/6,6/7,7/8,8/9,9/10,10/11,11/12,12/13,13/14,14/15,15/16}
    \foreach \ya/\yb in {3/3}
   {
     \draw [->] (\ya-\xa)--(a-\xa);
     \draw [->] (a-\xa) --(\ya-\xb);
   }     
 \draw [dashed] (0,0.4) -- (0,7.6); 
 \draw [dashed] (32,0.4) -- (32,7.6);
\end{tikzpicture} } 
\end{equation}
\begin{equation}
 \scalebox{.7}{ \begin{tikzpicture}[baseline=-1.7cm,scale=.49,yscale=-1]\label{E8_P3}
 \fill[fill1] (1,0.5)--(4,3.5)--(7,0.5)--cycle;
 \fill[fill1] (0,7.5)--(0,6.5)--(4,2.5)--(9,7.5)--cycle;
 \fill[fill1] (31.5,7)--(32,6.5)--(32,7.5)--cycle;
 \fill[fill1] (2.5,3)--(3,2.5)--(3.5,3)--(3,3.5)--cycle;
 \fill[fill1] (4.5,3)--(5,2.5)--(5.5,3)--(5,3.5)--cycle; 
 \foreach \x in {0,...,16}
  \foreach \y in {1,3,5,7}
   \node (\y-\x) at (\x*2,\y) [vertex] {};
 \foreach \x in {0,...,15}
  \foreach \y in {2,4,6}
   \node (\y-\x) at (\x*2+1,\y) [vertex] {};
 \foreach \x in {0,...,15}
  \foreach \y in {3}
   \node (a-\x) at (\x*2+1,\y) [vertex] {};
   
   \replacevertex{(3-2)}{[tvertex][fill1] {$M$}} 
 \foreach \xa/\xb in {0/1,1/2,2/3,3/4,4/5,5/6,6/7,7/8,8/9,9/10,10/11,11/12,12/13,13/14,14/15,15/16}
  \foreach \ya/\yb in {1/2,3/2,3/4,5/4,5/6,7/6}
   {
    \draw [->] (\ya-\xa) -- (\yb-\xa);
    \draw [->] (\yb-\xa) -- (\ya-\xb);
   }
 \foreach \xa/\xb in {0/1,1/2,2/3,3/4,4/5,5/6,6/7,7/8,8/9,9/10,10/11,11/12,12/13,13/14,14/15,15/16}
    \foreach \ya/\yb in {3/3}
   {
     \draw [->] (\ya-\xa)--(a-\xa);
     \draw [->] (a-\xa) --(\ya-\xb);
   }    
 \draw [dashed] (0,0.4) -- (0,7.6); 
 \draw [dashed] (32,0.4) -- (32,7.6);
\end{tikzpicture} } 
\end{equation}
%
\begin{equation}
 \scalebox{.7}{ \begin{tikzpicture}[baseline=-1.7cm,scale=.49,yscale=-1]\label{E8_P2}
 \fill[fill1] (3,0.5)--(5,2.5)--(7,0.5)--cycle;
 \fill[fill1] (0,7.5)--(0,6.5)--(5,1.5)--(11,7.5)--cycle;
 \fill[fill1] (0,0.5)--(0.5,1)--(0,1.5)--cycle;
 \fill[fill1] (31.5,1)--(32,0.5)--(32,1.5)--cycle;
 \fill[fill1] (31.5,7)--(32,6.5)--(32,7.5)--cycle;
 \fill[fill1] (2.5,3)--(3,2.5)--(3.5,3)--(3,3.5)--cycle;
 \fill[fill1] (6.5,3)--(7,2.5)--(7.5,3)--(7,3.5)--cycle;
 \fill[fill1] (15.5,7)--(16,6.5)--(16.5,7)--(16,7.5)--cycle;
 \fill[fill1] (25.5,7)--(26,6.5)--(26.5,7)--(26,7.5)--cycle;
 \fill[fill1] (9.5,1)--(10,0.5)--(10.5,1)--(10,1.5)--cycle;
 \foreach \x in {0,...,16}
  \foreach \y in {1,3,5,7}
   \node (\y-\x) at (\x*2,\y) [vertex] {};
 \foreach \x in {0,...,15}
  \foreach \y in {2,4,6}
   \node (\y-\x) at (\x*2+1,\y) [vertex] {};
 \foreach \x in {0,...,15}
  \foreach \y in {3}
   \node (a-\x) at (\x*2+1,\y) [vertex] {};
   
   \replacevertex{(2-2)}{[tvertex][fill1] {$M$}} 
 \foreach \xa/\xb in {0/1,1/2,2/3,3/4,4/5,5/6,6/7,7/8,8/9,9/10,10/11,11/12,12/13,13/14,14/15,15/16}
  \foreach \ya/\yb in {1/2,3/2,3/4,5/4,5/6,7/6}
   {
    \draw [->] (\ya-\xa) -- (\yb-\xa);
    \draw [->] (\yb-\xa) -- (\ya-\xb);
   }
 \foreach \xa/\xb in {0/1,1/2,2/3,3/4,4/5,5/6,6/7,7/8,8/9,9/10,10/11,11/12,12/13,13/14,14/15,15/16}
    \foreach \ya/\yb in {3/3}
   {
     \draw [->] (\ya-\xa)--(a-\xa);
     \draw [->] (a-\xa) --(\ya-\xb);
   }     
 \draw [dashed] (0,0.4) -- (0,7.6); 
 \draw [dashed] (32,0.4) -- (32,7.6);
\end{tikzpicture} } 
\end{equation}
%
\begin{equation}
 \scalebox{.7}{ \begin{tikzpicture}[baseline=-1.7cm,scale=.49,yscale=-1]\label{E8_P1}
 \fill[fill1] (0,7.5)--(0,6.5)--(6,0.5)--(13,7.5)--cycle;
 \fill[fill1] (9,0.5)--(11,2.5)--(13,0.5)--cycle;
 \fill[fill1] (0,0.5)--(0,1.5)--(1,2.5)--(3,0.5)--cycle;
 \fill[fill1] (31.5,1)--(32,0.5)--(32,1.5)--cycle;
 \fill[fill1] (31.5,7)--(32,6.5)--(32,7.5)--cycle;
 \fill[fill1] (2.5,3)--(3,2.5)--(3.5,3)--(3,3.5)--cycle;
 \fill[fill1] (12.5,3)--(13,2.5)--(13.5,3)--(13,3.5)--cycle;
 \fill[fill1] (30.5,3)--(31,2.5)--(31.5,3)--(31,3.5)--cycle;
 \fill[fill1] (8.5,3)--(9,2.5)--(9.5,3)--(9,3.5)--cycle;  
 \fill[fill1] (15,7.5)--(17,5.5)--(19,7.5)--cycle;
 \fill[fill1] (25,7.5)--(27,5.5)--(29,7.5)--cycle;
 \fill[fill1] (15.5,1)--(16,0.5)--(16.5,1)--(16,1.5)--cycle;
 \fill[fill1] (21.5,1)--(22,0.5)--(22.5,1)--(22,1.5)--cycle;
 \fill[fill1] (27.5,1)--(28,0.5)--(28.5,1)--(28,1.5)--cycle;
 \fill[fill1] (21.5,7)--(22,6.5)--(22.5,7)--(22,7.5)--cycle;
 \foreach \x in {0,...,16}
  \foreach \y in {1,3,5,7}
   \node (\y-\x) at (\x*2,\y) [vertex] {};
 \foreach \x in {0,...,15}
  \foreach \y in {2,4,6}
   \node (\y-\x) at (\x*2+1,\y) [vertex] {};
 \foreach \x in {0,...,15}
  \foreach \y in {3}
   \node (a-\x) at (\x*2+1,\y) [vertex] {}; 
   \replacevertex{(1-3)}{[tvertex][fill1] {$M$}} 
 \foreach \xa/\xb in {0/1,1/2,2/3,3/4,4/5,5/6,6/7,7/8,8/9,9/10,10/11,11/12,12/13,13/14,14/15,15/16}
  \foreach \ya/\yb in {1/2,3/2,3/4,5/4,5/6,7/6}
   {
    \draw [->] (\ya-\xa) -- (\yb-\xa);
    \draw [->] (\yb-\xa) -- (\ya-\xb);
   }
 \foreach \xa/\xb in {0/1,1/2,2/3,3/4,4/5,5/6,6/7,7/8,8/9,9/10,10/11,11/12,12/13,13/14,14/15,15/16}
    \foreach \ya/\yb in {3/3}
   {
     \draw [->] (\ya-\xa)--(a-\xa);
     \draw [->] (a-\xa) --(\ya-\xb);
   }      
 \draw [dashed] (0,0.4) -- (0,7.6); 
 \draw [dashed] (32,0.4) -- (32,7.6);
\end{tikzpicture} } 
\end{equation}
\begin{equation}
 \scalebox{.7}{ \begin{tikzpicture}[baseline=-1.7cm,scale=.49,yscale=-1]\label{E8_P8}
 \fill[fill1] (3,0.5)--(7,4.5)--(11,0.5)--cycle;
 \fill[fill1] (1,7.5)--(7,1.5)--(13,7.5)--cycle;
 \fill[fill1] (0,0.5)--(0.5,1)--(0,1.5)--cycle;
 \fill[fill1] (31.5,1)--(32,0.5)--(32,1.5)--cycle;
 \fill[fill1] (2.5,3)--(3,2.5)--(3.5,3)--(3,3.5)--cycle;
 \fill[fill1] (10.5,3)--(11,2.5)--(11.5,3)--(11,3.5)--cycle;
 \fill[fill1] (15.5,7)--(16,6.5)--(16.5,7)--(16,7.5)--cycle;
 \fill[fill1] (19.5,7)--(20,6.5)--(20.5,7)--(20,7.5)--cycle;
 \fill[fill1] (25.5,7)--(26,6.5)--(26.5,7)--(26,7.5)--cycle;
 \fill[fill1] (29.5,7)--(30,6.5)--(30.5,7)--(30,7.5)--cycle;
 \fill[fill1] (13.5,1)--(14,0.5)--(14.5,1)--(14,1.5)--cycle;
 \foreach \x in {0,...,16}
  \foreach \y in {1,3,5,7}
   \node (\y-\x) at (\x*2,\y) [vertex] {};
 \foreach \x in {0,...,15}
  \foreach \y in {2,4,6}
   \node (\y-\x) at (\x*2+1,\y) [vertex] {};
 \foreach \x in {0,...,15}
  \foreach \y in {3}
   \node (a-\x) at (\x*2+1,\y) [vertex] {};
   
   \replacevertex{(a-3)}{[tvertex][fill1] {$M$}} 
 \foreach \xa/\xb in {0/1,1/2,2/3,3/4,4/5,5/6,6/7,7/8,8/9,9/10,10/11,11/12,12/13,13/14,14/15,15/16}
  \foreach \ya/\yb in {1/2,3/2,3/4,5/4,5/6,7/6}
   {
    \draw [->] (\ya-\xa) -- (\yb-\xa);
    \draw [->] (\yb-\xa) -- (\ya-\xb);
   }
 \foreach \xa/\xb in {0/1,1/2,2/3,3/4,4/5,5/6,6/7,7/8,8/9,9/10,10/11,11/12,12/13,13/14,14/15,15/16}
    \foreach \ya/\yb in {3/3}
   {
     \draw [->] (\ya-\xa)--(a-\xa);
     \draw [->] (a-\xa) --(\ya-\xb);
   }      
 \draw [dashed] (0,0.4) -- (0,7.6); 
 \draw [dashed] (32,0.4) -- (32,7.6);
\end{tikzpicture} } 
\end{equation}\\
\end{proof}

\section{The Euclidean Case }\label{seksjonTilde}
In this section we consider cluster-tilted algebras of Euclidean type. It was shown in ~\cite{ABGP} that cluster-tilted algebras of type $\widetilde{A}$ are gentle, and hence special biserial. In the following proposition we consider the cluster-tilted algebras of type $\widetilde{D}_n$,$\widetilde{E}_6$,$\widetilde{E}_7$ and $\widetilde{E}_8$.
\begin{proposition}\label{tildeprop}
Let $T$ be a tilting object in $\mathcal{C}_{\widetilde{D}_n}$,$\mathcal{C}_{\widetilde{E}_6}$,$\mathcal{C}_{\widetilde{E}_7}$ or $\mathcal{C}_{\widetilde{E}_8}$. Then $\Gamma=\End(T)^{\op}$ is not special biserial.
\end{proposition}
\begin{proof}
Let $\mathcal{C}$ be a cluster category of type $\widetilde{D}_n$,$\widetilde{E}_6$,$\widetilde{E}_7$ or $\widetilde{E}_8$. Path algebras of these types are of infinite representation type, and so 
the transjective component of the AR-quiver of $\mathcal C$ is of the form $\mathbb{Z}Q$, where $Q$ is any orientation of either $\widetilde{D}_n$,$\widetilde{E}_6$,$\widetilde{E}_7$ or $\widetilde{E}_8$. By ~\cite{BMRRT} any tilting object in $\mathcal C$ has at least one indecomposable summand in the transjective component. Furthermore if $M$ is an indecomposable object in the transjective component of $\mathcal C$, then by ~\cite[lemma 3.3]{BI} there are only finitely many indecomposable objects $X$ in $\mathcal C$ such that $\Ext^1_{\mathcal{C}}(X,M)=0$. By the Serre duality formula there is then only finitely many indecomposable objects $Z$ in $\mathcal C$ such that $\Hom_{\mathcal C}(M,Z)=0$. Thus for any tilting object $T$ having $M$ as a direct summand we can find an indecomposable object $Y$ satisfying the assumptions of Lemma \ref{lemma1}.
\end{proof}
From this proposition and \cite{ABGP}, we have the following:
\begin{corollary}
If $\Gamma$ is a cluster-tilted algebra of Euclidean type, then $\Gamma$ is special biserial algebra if and only if $H$ is of type $\widetilde{A}$.
\end{corollary}

\section{Summary}
It has been shown in \cite[corollary 2.4]{WaldWasch} that any special biserial algebra is either of finite representation type or tame. By \cite[corollary 3.3]{BZ}, if $T$ is a tilting object in the cluster category of a finite dimensional hereditary algebra $H$ and $\Gamma=\End_{\mathcal{C}_H}(T)^{\op}$ is the cluster-tilted algebra of $T$, then $\Gamma$ is of finite representation type if and only if $H$ is of finite representation type and $\Gamma$ is tame if and only if $H$ is tame. Combining this with Theorem ~\ref{typeDiff} and Propositions ~\ref{E6prop},~\ref{E7prop},~\ref{E8prop} and ~\ref{tildeprop} we have the following description of special biserial cluster-tilted algebras:

\begin{theorem}
Let $H$ be a finite dimensional hereditary algebra over an algebraically closed field $K$ and $T$ be a basic cluster-tilting object in the cluster category $\mathcal{C}_{H}$. Then $\Gamma=\End_{\mathcal C}(T)^{\op}$ is a special biserial algebra if and only if 
\begin{itemize}
\item[a)] $H$ is of type $A$, or
\item[b)] $H$ is of type $\widetilde{A}$, or
\item[c)] $H$ is of type $D$, and any indecomposable summand in $T$ is either an $\alpha$- or a $\gamma$-object. 
\end{itemize}
\end{theorem}

\bibliographystyle{plain}

\end{document}